\documentclass[aos]{imsart}
\usepackage{graphicx}
\usepackage{amssymb,amsmath,amsthm}
\usepackage{multirow}
\usepackage{natbib}
\usepackage{rotating}
\usepackage[english]{babel}
\usepackage[pdftex,plainpages=false,pdfpagelabels,pagebackref=true,citecolor=blue,urlcolor=blue,colorlinks=true,pdfpagemode=UseOutlines]{hyperref}

\startlocaldefs
\def\btau{\boldsymbol{\tau}}

\def\bZ{\boldsymbol{Z}}
\def\bz{\boldsymbol{z}}

\def\ss{\boldsymbol{s}}

\newcommand{\U}{U}
\newcommand{\e}{\text{e}}

\makeatletter
\def\varphiall{\@ifnextchar[{\varphiall@i}{\varphiall@i[]}}
\def\varphiall@i[#1]{\@ifnextchar[{\varphiall@ii{#1}}{\varphiall@ii{#1}[#1]}}
\def\varphiall@ii#1[#2]{\varphi_{\mu,\kappa,\beta_{#1},\sigma^2_{#2}}}
\def\hatvarphiall{\@ifnextchar[{\hatvarphiall@i}{\hatvarphiall@i[]}}
\def\hatvarphiall@i[#1]{\@ifnextchar[{\hatvarphiall@ii{#1}}{\hatvarphiall@ii{#1}[#1]}}
\def\hatvarphiall@ii#1[#2]{\widehat{\varphi}_{\mu,\kappa,\beta_{#1},\sigma^2_{#2}}}
\makeatother
\def\varphimukappa{\varphi_{\mu,\kappa}}

\def\R{\mathbb{R}}

\def\d{\textrm{d}}

\newtheorem{theo}{Theorem}

\newtheorem{lemma}{Lemma}

\renewcommand{\c}{\boldsymbol{c}}

\DeclareMathOperator*{\argmax}{arg\,max}
\DeclareMathOperator{\var}{Var}
\newcommand{\tlambda}{\tilde{\lambda}}
\endlocaldefs

\usepackage{color}
\definecolor{rf}{rgb}{1, 0.49, 0}

\begin{document}

\begin{frontmatter}
\title{Estimation and Prediction
 using generalized Wendland  Covariance Functions under fixed domain asymptotics}

\runtitle{Generalized Wendland  Covariance Functions}

\begin{aug}
    \author{\fnms{Moreno} \snm{Bevilacqua}\corref{}\thanksref{t1,m1}\ead[label=e1]{moreno.bevilacqua@uv.cl}},
  \author{\fnms{Tarik} \snm{Faouzi}\thanksref{m2}\ead[label=e2]{tfaouzi@ubiobio.cl}}
  \author{\fnms{Reinhard} \snm{Furrer}\thanksref{t3,m3}\ead[label=e3]{reinhard.furrer@math.uzh.ch}}%
  \and
  \author{\fnms{Emilio} \snm{Porcu}\thanksref{t4,m4}\ead[label=e4]{emilio.porcu@usm.cl}%
}
\runauthor{M. Bevilacqua et al.}
\thankstext{t1}{Supported by grant FONDECYT 1160280 from  Chilean government}
\thankstext{t3}{Acknowledges support of URPP-GCB and SNSF-175529}
\thankstext{t4}{Supported by grant FONDECYT 1130647 from  Chilean government}

  \affiliation{Universidad de Valparaiso\thanksmark{m1} Universidad del BioBio\thanksmark{m2}
  University of Zurich\thanksmark{m3} Universidad Federico Santa Maria\thanksmark{m4}, University of Newcastle\thanksmark{m4}}

\address{Department of Statistics\\ Universidad de Valparaiso\\
  2360102 Valparaiso, Chile\\ \printead{e1}}

  \address{Department of Statistics\\ Universidad del BioBio\\ 4081112
    Concepcion, Chile\\
    \printead{e2}\\\phantom{Universidad Federico Santa Maria}}

  \address{Department of Mathematics and\\ Department of Computational
    Science\\ University of Zurich\\ 8057 Zurich, Switzerland\\
    \printead{e3}}

  \address{Department of Mathematics\\ Universidad Federico Santa
    Maria\\ 2360102 Valparaiso, Chile\\ \printead{e4}}
    
\end{aug}

\begin{abstract} \hspace{0.1cm}    We study estimation and prediction of Gaussian random fields with
covariance models belonging to the generalized Wendland (GW) class,
under fixed domain asymptotics.  As for the Mat{\'e}rn case, this class
allows for a continuous parameterization of smoothness of the underlying
Gaussian random field, being additionally compactly supported.
The paper is divided into three parts: first, we
characterize the equivalence of two Gaussian measures with GW
covariance function, and we provide sufficient conditions for the
equivalence of two Gaussian measures with Mat{\'e}rn and GW covariance
functions. 
In the second part,
we establish strong consistency and asymptotic distribution of the
maximum likelihood estimator of the microergodic parameter associated
to GW covariance model, under fixed domain asymptotics.
The third part elucidates the consequences of our results in terms 
of (misspecified) best linear unbiased predictor, under fixed domain asymptotics.
Our findings are illustrated through a simulation study: the former
compares the finite sample behavior of the maximum likelihood
estimation of the microergodic parameter with the given asymptotic
distribution.  The latter compares the finite-sample behavior of the
prediction and its associated mean square error when using two
equivalent Gaussian measures with Mat{\'e}rn and GW covariance models,
using covariance tapering as benchmark.
\end{abstract}



\end{frontmatter}

\section{Introduction}
\def\blu{\textcolor{green}}
Covariance functions cover a central aspect of inference and
prediction of Gaussian fields defined over some (compact) set of
$\R^d$.  For instance, the best linear unbiased prediction at an
unobserved site depends on the knowledge of the covariance function.
Since a covariance function must be positive definite, practical
estimation generally requires the selection of some parametric classes
of covariances and the corresponding estimation of these parameters.

The maximum likelihood (ML) estimation method is generally considered
best for estimating the parameters of covariance models.
The study of asymptotic properties of ML estimators
is complicated by the fact that more than one asymptotic framework
can be considered when observing a single realization from a Gaussian
field. In particular, under fixed domain asymptotics, one supposes
that the sampling domain is bounded and that the sampling set becomes
increasingly dense. Under increasing domain asymptotics, the sampling domain
increases with the number of observed data, and the distance between
any two sampling locations is bounded away from zero.

The asymptotic behavior of ML estimators of the covariance parameters
can be quite different under these two frameworks
\citep{Zhang:Zimmerman:2005}. 
Under increasing domain asymptotic framework and some mild regularity conditions \cite{Mardia:Marshall:1984} give a general result. Specifically, they show that ML
estimators are consistent and asymptotically Gaussian, with asymptotic
covariance matrix equal to the inverse of the Fisher information
matrix.

Equivalence of Gaussian measures
\citep{Sko:ya:1973,Ibragimov-Rozanov:1978} represents an essential tool
to establish the asymptotic properties of both prediction and
estimation of Gaussian fields under fixed domain asymptotics.  In his
{\em tour de force}, \cite{Stein:1988,Stein:1990,Stein:1993,Stein:1999b,Stein:2004}
provides conditions under which predictions under a misspecified
covariance function are asymptotically efficient, and mean square
errors converge almost surely to their targets. Since Gaussian
measures depend exclusively on their mean and covariance functions,
practical evaluation of Stein's conditions translates into the fact
that the true and the misspecified covariances must be compatible,
i.e., the induced Gaussian measures are equivalent.

Under fixed domain asymptotics, no general results are available for
the asymptotic properties of ML estimators. Yet, some results have
been obtained when assuming that the covariance belongs to the
parametric family of Mat{\'e}rn covariance functions \citep{Mate:60}
that has been very popular in spatial statistics for its flexibility
with respect to continuous parameterization of smoothness, in the mean
square sense, of the underlying Gaussian field. For a Gaussian field defined over a bounded and infinite set of $\mathbb{R}^d$, \cite{Zhang:2004} shows that when the smoothness
parameter is known and fixed, not all parameters can be estimated
consistently when $d=1, 2, 3$. Instead, the ratio of variance and
scale (to the power of the smoothing parameter), sometimes called
microergodic parameter \citep{Zhang:Zimmerman:2005,Stein:1999}, is
consistently estimable.
In contrast for $d \geq 5$, \cite{anderes2010} proved  the  orthogonality of two Gaussian measures
with different Mat{\'e}rn covariance functions
and hence, in this case, all the parameters can be consistently estimated under fixed-domain asymptotics. The case $d = 4$
is still open.
%


Asymptotic results for ML estimator of the microergodic parameter of
the Mat{\'e}rn model can be found in \cite{Zhang:2004},
\cite{Du:Zhang:Mandrekar:2009}, \cite{Wang:Loh:2011} and
\cite{Shaby:Kaufmann:2013}.  In particular, \cite{Shaby:Kaufmann:2013}
give strong consistency and asymptotic distribution of the
microergodic parameter when estimating jointly the scale and
variance parameters, generalizing previous results in \cite{Zhang:2004}
and \cite{Wang:Loh:2011} where the scale parameter is assumed known
and fixed.  \cite{Shaby:Kaufmann:2013} show by means of simulation
study that asymptotic approximation using a fixed scale parameter
can be problematic when applied to finite samples, even for large
sample sizes. In contrast, they show that performance is improved and
asymptotic approximations are applicable for smaller sample sizes when
the parameters are jointly estimated.

Under the Mat{\'e}rn family, similar results have been obtained for
the covariance tapering (CT) method of estimation, as originally
proposed in \cite{Kaufman:Schervish:Nychka:2008} and consisting of
setting to zero the dependence after a given distance. This is in turn
achieved by multiplying the Mat{\'e}rn covariance with a taper
function, that is, a correlation function being additionally compactly
supported over a ball with given radius. Thus, the resulting
covariance tapered matrix is sparse, with the level of sparseness
depending on the radius of compact support.  Sparse matrix
algorithms can then be used to evaluate efficiently an approximate
likelihood where the original covariance matrix is replaced by the
tapered matrix.  The results proposed in
\cite{Kaufman:Schervish:Nychka:2008} have then inspired the works in
\cite{Du:Zhang:Mandrekar:2009}, \cite{Wang:Loh:2011} and
\cite{Shaby:Kaufmann:2013}, where asymptotic properties of the CT
estimator of the Mat{\'e}rn microergodic parameter are given.

Using the Mat{\'e}rn family, \cite{Furrer:2006}
study CT when applied to the best unbiased linear predictor and show
that under fixed domain asymptotics and some specific conditions of
the taper function, asymptotically efficient prediction and
asymptotically correct estimation of mean square error can be achieved
using a tapered Mat{\'e}rn covariance function.
Extensions have been discussed by, e.g.,  \cite{Stei:13} and \cite{Hirano2013}.
The basic message of CT method is that large data sets can be handled
both for estimation and prediction exploiting sparse matrix
algorithms when using the Mat{\'e}rn model.

Inspired by this idea, we focus on a covariance model that offers the
strength of the Mat{\'e}rn family and allows the use of sparse
matrices.  Specifically, we study estimation and prediction of
Gaussian fields under fixed domain asymptotics, using the generalized
Wendland (GW) class of covariance functions
\citep{Gneiting:2002b,Zastavnyi2006}, the  members of which are compactly supported
over balls of $\R^d$ with arbitrary radii, and additionally allows for
a continuous parameterization of differentiability at the origin, in a
similar way to the Mat{\'e}rn family.

In particular, we provide the following results.  First, we characterize
the equivalence of two Gaussian measures with covariance functions
belonging to the GW class and sharing the same smoothness parameter.
A consequence of this result is that, as in the Mat{\'e}rn case
\citep{Zhang:2004}, when the smoothness parameter is known and fixed,
not all parameters can be estimated consistently under fixed domain
asymptotics.  Then we give sufficient conditions for the equivalence
of two Gaussian measures where the state of truth is represented by a
member of the Mat{\'e}rn family and the other measure has a GW
covariance model and vice versa.

We assess the asymptotic properties of the ML estimator of the
microergodic parameter associated to the GW class.  Specifically, for
a fixed smoothness parameter, we establish strong consistency and
asymptotic distribution of the microergodic parameter estimate assuming the
compact support parameter fixed and known.  Then, we generalize these
results when jointly estimating with ML the variance and the compact
support parameter.

Finally, using results in \cite{Stein:1988,Stein:1993}, we study the
implications of our results on prediction under fixed domain
asymptotics.  One remarkable implication is that when the true
covariance belongs to the Mat{\'e}rn family, asymptotic efficiency
prediction and asymptotically correct estimation of mean square error
can be achieved using a compatible GW covariance model.

The remainder of the paper is organized as follows.  In Section~\ref{2}, we
review some results of Mat{\'e}rn and GW covariance models.  In
Section~\ref{3}, we first characterize the equivalence of Gaussian measure
under the GW covariance model.  Then we find a sufficient condition
for the equivalence of two Gaussian measures with Mat{\'e}rn and GW
covariance models.  In Section~\ref{4}, we establish strong consistency and
asymptotic distribution of the ML estimator of the microergodic
parameter of the GW models, under fixed domain asymptotics.  Section~\ref{5}
discusses the consequences of the results in Section~\ref{3} on prediction
under fixed domain asymptotics.  Section~\ref{6} provides two simulation
studies: The first shows how well the given asymptotic distribution of
the microergodic parameter applies to finite sample cases when
estimating with ML a GW covariance model under fixed domain
asymptotics.  The second compare the finite-sample behavior of the
prediction when using two compatible Mat{\'e}rn and GW models, using
CT as a benchmark.  The final section provides discussion on the
consequence of our results and identifies problems for future research.

\section{Mat{\'e}rn and generalized Wendland covariance models}\label{2}

We denote $\{Z(\ss), \ss \in D \} $ a zero mean Gaussian field on a
bounded set $D$ of $\R^d$, with stationary covariance function $C:\R^d
\to \R$. We consider the class $\Phi_d$ of continuous mappings
$\phi:[0,\infty) \to \R$ with $\phi(0) >0$, such that
\begin{equation*} 
{\rm cov} \left ( Z(\ss), Z(\ss^{\prime}) \right )= C(\ss^{\prime}-\ss)=  \phi(\|\ss^{\prime } -\ss \|),
\end{equation*}
with $\ss,\ss^{\prime} \in D$, and $\|\cdot\|$ denoting the Euclidean norm. Gaussian fields with such covariance functions are called weakly stationary and isotropic.

\cite{Shoe38} characterized the class $\Phi_d$ as scale mixtures
of the characteristic functions of random vectors uniformly
distributed on the spherical shell of $\R^d$, with any   positive measure, $F$:
$$ \phi(r)= \int_{0}^{\infty} \Omega_{d}(r \xi) F(\d \xi), \qquad r \ge 0,$$
with $\Omega_{d}(r)= r^{1-d/2}J_{d/2-1}(r)$ and $J_{\nu}$ a Bessel function of order $\nu$.
The class $\Phi_d$ is nested, with the inclusion relation $\Phi_{1} \supset \Phi_2 \supset \ldots \supset \Phi_{\infty}$ being strict, and where $\Phi_{\infty}:= \bigcap_{d \ge 1} \Phi_d$ is the class of continuous mappings $\phi$, the radial version of which is positive definite on any $d$-dimensional Euclidean space.

The Mat{\'e}rn function, defined as:
\begin{equation*}
{\cal M}_{\nu,\alpha,\sigma^2}(r)=
  \sigma^2 \frac{2^{1-\nu}}{\Gamma(\nu)} \left (\frac{r}{\alpha}
  \right )^{\nu} {\cal K}_{\nu} \left (\frac{r}{\alpha} \right ),
  \qquad r \ge 0,
\end{equation*}
is a member of the class
$\Phi_{\infty}$ for any positive values of $\alpha$ and $\nu$. Here,
${\cal K}_{\nu}$ is a modified Bessel function of the second kind of
order $\nu$, $\sigma^2$ is the variance and $\alpha$ a positive
scaling parameter.
The parameter $\nu$ characterizes the
differentiability at the origin and, as a consequence,  the
differentiability of the associated sample paths. In particular for a
positive integer $k$, the sample paths are $k$ times differentiable, in any direction, if
and only if $\nu>k$.

When $\nu=1/2+m$ and $m$ is a nonnegative integer, the Mat{\'e}rn
function simplifies to the product of a negative exponential with a
polynomial of degree $m$, and for $\nu$ tending to infinity, a rescaled
version of the Mat{\'e}rn converges to a squared exponential model being
infinitely differentiable at the origin. Thus, the Mat{\'e}rn function
allows for a continuous parameterization of its associated Gaussian
field in terms of smoothness.

%
%
We also define $\Phi_{d}^{b}$ as the class that consists of members of $\Phi_d$ being additionally compactly supported on a given interval, $[0,b]$, $b>0$. Clearly, their radial versions are compactly supported over balls of $\R^d$ with radius $b$.


We now  define GW  correlation functions $\varphimukappa$ as introduced by \cite{Gneiting:2002}, \cite{Zastavnyi2006} and  \cite{Hubb14}. For $\kappa>0$, we define
\begin{equation} \label{WG2*}
\varphimukappa(r):= \begin{cases}  \frac{1}{B(2\kappa,\mu+1)} \int_r^{1} u(u^2-r^2)^{\kappa-1} (1-u)^{\mu}\,\d u  ,& 0 \leq r < 1,\\ 0,&r \geq 1, \end{cases}
\end{equation}
with $B$ denoting the beta function.  Arguments in \cite{Gneiting:2002} and  \cite{Zastavnyi2006} show that, for a given $\kappa>0$, $\varphimukappa \in \Phi^1_d$
if and only if
\begin{equation} \label{lawea}
\mu \ge \lambda(d,\kappa):= (d+1)/2+ \kappa.
\end{equation}
Throughout, we use $\lambda$ instead of $\lambda(d,\kappa)$ whenever no confusion arises.
Integration by parts shows that the first part of \eqref{WG2*} can also
 be written as:
\begin{equation*} 
\frac{1}{B( 1+2\kappa,\mu)} \int_r^{1} (u^2-r^2)^{\kappa} (1-u)^{\mu-1}\,\d u,\qquad  0 \leq r < 1. 
 \end{equation*}
Note that  $\varphi_{\mu,0}$ is not defined because $\kappa$ must be strictly positive. In this special case 
we consider
the Askey function \citep{Askey:1973} 
\begin{equation*} 
{\cal A}_{\mu}(r)=\left ( 1- r \right )_{+}^{\mu} =   \begin{cases}   \left ( 1- r \right )^{\mu} ,& 0 \leq r < 1,\\ 0,&r \geq 1, \end{cases}
\end{equation*}
where $(\cdot)_+$ denotes the positive part.
Arguments in \cite{Golubov:1981} show that ${\cal A}_{\mu} \in \Phi_d^1$ if and only if $\mu \ge (d+1)/2$
and   we define $\varphi_{\mu,0}:= {\cal A}_{\mu}$.

Finally, we define the GW covariance function, with compact support $\beta>0$, variance $\sigma^2$  and
smoothness parameter $\kappa >0$ as:
\begin{equation}\label{wendgen}
\varphiall(r):= \sigma^2 \varphimukappa(r/\beta),  \quad r \ge 0,
\end{equation}
and $\varphiall\in \Phi_d^{\beta}$ if and only if  $\mu \ge \lambda$. Accordingly, for $\kappa=0$, we define
\begin{equation}\label{wendgen2}
\varphi_{\mu,0,\beta,\sigma^2}(r):=\sigma^2 \varphi_{\mu,0} (r/\beta), \quad  r \ge 0.
\end{equation}

When computing~\eqref{wendgen}, numerical integration is obviously
feasible, but could be cumbersome to (spatial) statisticians used to
handle closed form parametric covariance model.  Nevertheless, closed
form solution of the integral in Equation~\eqref{WG2*} can be obtained when
$\kappa=k$, a positive integer.  In this  case,
 $ \varphi_{\mu,k,1,1}(r) =
{\cal A}_{\mu+k}(r) P_{k}(r)$,
with $P_{k}$ a polynomial of order $k$.
These functions, termed (original) Wendland functions, were originally proposed by   \cite{Wendland:1995}.
Other closed form solutions of integral~\eqref{WG2*} can be obtained when $\kappa=k+0.5$, using some results in \cite{Schaback:2011}. Such solutions are called 
 \emph{missing Wendland}  functions.

  Recently, \cite{Porcu:Zastavnyi:Xesbaiat} have shown that the GW class includes almost all classes of covariance functions with compact supports known to the geostatistical and numerical analysis communities. Not only original and Wendland functions, but also Wu's functions \citep{wu95}, which in turn include the spherical model \citep{Wackernagel:2003}, as well as the Trigub splines \citep{Zastavnyi2006}.  Finally, \cite{che14}  show that,   for $\kappa$ tending to infinity, a rescaled   version of the  GW model converges  to 
 a squared exponential  covariance model.

 As noted by  \cite{Gneiting:2002b}, GW and Mat{\'e}rn functions exhibit the same behavior at the origin,
with the smoothness parameters of the two covariance models related by the equation $\nu=\kappa+1/2$.

\begin{table}[t!]
\caption{GW correlation  $\varphi_{\mu,\kappa,1,1}(r)$ and Mat{\'e}rn correlation  ${\cal M}_{\nu,1,1}(r)$  with increasing smoothness parameters $\kappa$ and $\nu$.
$SP(k)$ means that the sample paths of the  associated  Gaussian field
are  $k$ times differentiable. }\label{tab1}
\centering
{\scriptsize\tabcolsep2pt \linespread{2}\selectfont\begin{tabular}{|c|l|c|l|c|}
\hline
~$\kappa$~& $\varphi_{\mu,\kappa,1,1}(r)$ &  $\nu$&${\cal M}_{\nu,1,1}(r)$ &  $SP(k)$
\raisebox{0pt}[18pt][-7pt]{} \\
\hline
$0$ & $(1-r)^{\mu}_{+}$ & $0.5$&   $e^{-r}$ & 0\\
\hline
$1$ & $(1-r)^{\mu+1}_{+}(1+r(\mu+1))$ & $1.5$ & $ e^{-r}(1+r)$ &1 \\
\hline
$2$ & $(1-r)^{\mu+2}_{+}(1+r(\mu+2)+r^2(\mu^2+4\mu+3)\frac{1}{3})$ &$2.5$ &$ e^{-r}(1+r+\frac{r^{2}}{3})$&2\\
\hline
\multirow{2}{*}{$3$} & $(1-r)^{\mu+3}_{+}
\big( 1+r(\mu+3)+r^2(2\mu^2+12\mu+15)\frac{1}{5}\qquad $
&\multirow{2}{*}{$3.5$}&\multirow{2}{*}{$ e^{-r}(1+\frac{r}{2}+r^2 \frac{6}{15}+\frac{r^3}{15})$}&\multirow{2}{*}{$3$}\\
& $\hspace*{2.8cm} +r^3(\mu^3+9\mu^2+23\mu+15)\frac{1}{15}\big)\,$ &&&\\
\hline
\end{tabular} }
\end{table}

Fourier transforms of radial versions of members of $\Phi_d$, for a given $d$, have a simple expression, as reported in \citet{Yaglom:1987} or \citet{Stein:1999}. For a member $\phi$ of the class $\Phi_d$, we define its isotropic spectral density as
\begin{equation} \label{FT}
 \widehat{\phi}(z)= \frac{z^{1-d/2}}{(2 \pi)^d} \int_{0}^{\infty} u^{d/2} J_{d/2-1}(uz)  \phi(u) {\rm d} u, \qquad z \ge 0,
\end{equation}
and throughout the paper, we use the notations: $\widehat{{\cal M}}_{\nu,\alpha,\sigma^2}$, and $\hatvarphiall$ for the
radial parts of  Fourier transforms
of ${{\cal M}}_{\nu,\alpha,\sigma^2}$ and $\varphiall $, respectively.

A well-known result about the spectral density of the  Mat{\'e}rn model is the following:
\begin{equation} \label{stein1}
\widehat{{\cal M}}_{\nu,\alpha,\sigma^2}(z)= \frac{\Gamma(\nu+d/2)}{\pi^{d/2} \Gamma(\nu)}
\frac{\sigma^2 \alpha^d}{(1+\alpha^2z^2)^{\nu+d/2}}
, \qquad z \ge 0.
\end{equation}
For two given non negative    functions $g_1(x)$ and $g_2(x)$, with $g_1(x) \asymp g_2(x)$ we mean
that there exist two constants $c$ and $C$
such that $0<c< C<\infty$ and $cg_2(x) \leq g_1(x) \leq Cg_2(x)$ for
each $x$.
The next result follows from \cite{Zastavnyi2006}, \cite{Hubb14}, and from standard properties of Fourier transforms. Their proofs are thus omitted.
Let us first define the function  $\mathstrut_1 F_2$  as:
\begin{equation*}
\mathstrut_1 F_2(a;b,c;z)=\sum_{k=0}^{\infty}\frac{(a)_{k}z^{k}}{(b)_{k}(c)_{k}k!}, \qquad z \in\mathbb{R},
\end{equation*}
which is a special case of the generalized hypergeometric functions $\mathstrut_q F_p$ \citep{Abra:Steg:70}, with
$(q)_{k}=  \Gamma(q+k)/\Gamma(q)$ for $k \in  \mathbb{N}\cup\{ 0\}$, being the Pochhammer symbol.

\begin{theo}\label{the3} Let $\varphiall$ be the function defined at Equation~\eqref{wendgen} and 
let  $\lambda$ as defined through Equation~\eqref{lawea}.
Then, for  $\kappa,\sigma^2,\beta>0$ and $\mu \geq \lambda$:
\begin{enumerate}
\item\label{3.1}
$\displaystyle\hatvarphiall(z)=\sigma^{2}L^{\boldsymbol{\varsigma}}\beta^{d}\mathstrut_1 F_2\Big(\lambda;\lambda+\frac{\mu}{2},\lambda+\frac{\mu}{2}+\frac{1}{2};-\frac{(z\beta)^{2}}4\Big), \quad  z > 0;$\\
\item\label{3.2}
$\displaystyle\hatvarphiall(z)=\sigma^2L^{\boldsymbol{\varsigma}}\beta^{d}\Big[c_{3}^{\boldsymbol{\varsigma}}(z\beta)^{-2\lambda}\big\{1+\mathcal{O}(z^{-2})\big\}$\\
\hspace*{1mm} ~\hfill$\displaystyle+\,c_{4}^{\boldsymbol{\varsigma}}(z\beta)^{-(\mu+\lambda)}\big\{\cos(z\beta-c_{5}^{\boldsymbol{\varsigma}})+\mathcal{O}(z^{-1})\big\}\Big]$,~
for $z\to\infty$;\\
\item\label{3.3} $\hatvarphiall(z)\asymp z^{-2\lambda} $,~  for $z\to\infty,$
\end{enumerate}
where $c_{3}^{\boldsymbol{\varsigma}}=\frac{\Gamma(\mu+2\lambda)}{\Gamma(\mu)}$, $c_{4}^{\boldsymbol{\varsigma}}=\frac{\Gamma(\mu+2\lambda)}{\Gamma(\lambda)2^{\lambda-1}}$, $c_{5}^{\boldsymbol{\varsigma}}=\frac{\pi}{2}(\mu+\lambda)$, ${L^{\boldsymbol{\varsigma}}=\frac{K^{\boldsymbol{\varsigma}}\Gamma(\kappa)}{2^{1-\kappa}B(2\kappa,\mu+1)}}$
and
\begin{equation*}
K^{\boldsymbol{\varsigma}}=\frac{2^{-\kappa-d+1}\pi^{-\frac{d}{2}}\Gamma(\mu+1)\Gamma(2\kappa+d)}{\Gamma(\kappa+\frac{d}{2})\Gamma(\mu+2\lambda)},
\end{equation*}
where $\boldsymbol{\varsigma}:= (\mu,\kappa,d)^{\prime}$.
\end{theo}

Point~\ref{3.1} has been shown by \cite{Zastavnyi2006}. Points~\ref{3.2} and~\ref{3.3} can be found in \cite{Hubb14}.
Note that the case $\kappa=0$ is not included in
Theorem~\ref{the3}. We consider it in the following result, whose
proof follows the lines of \cite{Zastavnyi2006} and \cite{Hubb14} for the case $\kappa>0$.

\begin{theo}\label{the356} Let $ \varphi_{\mu,0,\beta,\sigma^2}$ as being defined at Equation~\eqref{wendgen2}. Then, for  $\sigma^2,\beta>0$, $\mu \geq (d+1)/2$:
\begin{enumerate}
\item\label{4.1} 
$\displaystyle
 \widehat{ \varphi}_{\mu,0,\beta,\sigma^2}(z)=\sigma^{2}K^{\boldsymbol{\varsigma}}\beta^{d}$\\
\hspace*{1mm} ~\hfill$\displaystyle \times\
\mathstrut_1 F_2\left(\frac{d+1}2;\frac{d+1}2+\frac{\mu}{2},\frac{d+1}2+\frac{\mu}{2}+\frac{1}{2};-\frac{(z\beta)^{2}}4\right), \quad  z > 0$;\\
\item\label{4.2}  
$\displaystyle
 \widehat{ \varphi}_{\mu,0,\beta,\sigma^2}(z)=\sigma^{2}K^{\boldsymbol{\varsigma}}\beta^{d}\big[c_{3}^{\boldsymbol{\varsigma}}(z\beta)^{-(d+1)}\{1+\mathcal{O}(z^{-2})\}$\\
\hspace*{1mm} ~\hfill$\displaystyle+\,c_{4}^{\boldsymbol{\varsigma}}(z\beta)^{-(\mu+(d+1)/2)}\{\cos(z\beta-c_{5}^{\boldsymbol{\varsigma}})+\mathcal{O}(z^{-1})\}\big],
$~  
for $z\to\infty$;
\item\label{4.3}$ \widehat{ \varphi}_{\mu,0,\beta,\sigma^2}(z)\asymp z^{-(d+1)}$,~ for $ z\to\infty$,
\end{enumerate}
with  $c_3^{\boldsymbol{\varsigma}}$,    $c_4^{\boldsymbol{\varsigma}}$,   $c_5^{\boldsymbol{\varsigma}}$ and $K^{\boldsymbol{\varsigma}}$
defined as in Theorem~\ref{the3} but with $\boldsymbol{\varsigma}:= (\mu,0,d)^{\prime}$.
\end{theo}

The  spectral density  and its decay for $z\to \infty$ in Theorems~\ref{the3} and~\ref{the356} are  useful when studying some geometrical properties of a Gaussian field or its associated sample paths \citep{Adler:1981}.
For instance, using Theorem~\ref{the3} Point~\ref{3.1} or~Theorem~\ref{the356} Point~\ref{4.1},  it is easy to prove  that for a positive integer $k$, the sample paths  of a Gaussian field with GW function are $k$ times differentiable, in any direction, if and only if  $\kappa >k-1/2$.
Table~\ref{tab1}   compares the GW $\varphi_{\mu,\kappa,1,1}(r)$  for $\kappa=0,1,2,3$  with 
${\cal M}_{\nu,1,1}(r)$ for $\nu=0.5,1.5,2.5,3.5$ with the associated  degree of sample paths differentiability.


\section{Equivalence of Gaussian measures with GW models}\label{3}

Equivalence and orthogonality of probability measures are useful tools when assessing the asymptotic properties of both prediction and estimation for Gaussian fields.
Denote with $P_i$, $i=0,1$, two probability measures defined on the same
 measurable space $\{\Omega, \cal F\}$. $P_0$ and $P_1$ are called equivalent (denoted $P_0 \equiv P_1$) if $P_1(A)=1$ for any $A\in \cal F$ implies $P_0(A)=1$ and vice versa. On the other hand,  $P_0$ and $P_1$ are orthogonal (denoted $P_0 \perp P_1$) if there exists an event $A$ such that $P_1(A)=1$ but $P_0(A)=0$. For a stochastic process $\{ Z(\ss), \ss \in \R^d \}$, to define previous concepts, we restrict the event $A$ to the $\sigma$-algebra generated by $\{Z(\ss), \ss\in D\}$, where $D \subset \R^d$. We emphasize this restriction by saying that the
two measures are equivalent on the paths of $\{Z(\ss), \ss\in D\}$.

 Gaussian measures are completely characterized by their mean and covariance function.
We write $P(\rho)$ for a Gaussian measure with zero mean and covariance function $\rho$.  It is well known that two Gaussian measures  are either equivalent or orthogonal on the paths of $\{Z(\ss), \ss\in D\}$ \citep{Ibragimov-Rozanov:1978}.

Let $P(\rho_i)$, $i=0, 1$ be two zero mean Gaussian measures with  isotropic covariance function  $\rho_i$ and  associated spectral density $\widehat{\rho}_i$,  $i=0, 1$, as defined through~\eqref{FT}.
Using results in  \cite{Sko:ya:1973} and  \cite{Ibragimov-Rozanov:1978},  \cite{Stein:2004}
has shown that, if for some $a>0$,
$\widehat{\rho}_0(z)z^a$ is bounded away from 0 and $\infty$ as $z \to \infty$, and
for some finite and positive~$c$,
\begin{equation}\label{spectralfinite2}
\int_{c}^{\infty} z^{d-1} \;\left\{ \frac{\widehat{\rho}_1(z)-\widehat{\rho}_0(z)}{\widehat{\rho}_0(z)} \right\}^2\;
{\rm d} z <\infty,
\end{equation}
then for any bounded subset
$D\subset \R^d$, $P(\rho_0)\equiv P(\rho_1)$ on the paths of $Z(\ss), \ss \in D$.

For the reminder of the paper, we denote with $P({\cal M}_{\nu,\alpha,\sigma^2})$ a zero mean Gaussian measure induced by a Mat{\'e}rn covariance function
with associated spectral density $\widehat{{\cal M}}_{\nu,\alpha,\sigma^2}$,
 and
with $P( \varphiall)$
 a zero mean Gaussian measure induced by a GW covariance function with associated spectral density $\hatvarphiall$.

Using~\eqref{spectralfinite2} and  \eqref{stein1}, \cite{Zhang:2004} established the following characterization concerning the equivalent conditions
of two Gaussian measures with Mat{\'e}rn covariance models.

\begin{theo}[\citealp{Zhang:2004}] \label{Thm2}
For a given $\nu>0$, let $P({\cal M}_{\nu,\alpha_i,\sigma^2_i})$,  $i=0, 1$, be two zero mean Gaussian  measures. For any
bounded infinite set $D\subset \R^d$, $d=1, 2, 3$,
$P({\cal M}_{\nu,\alpha_0,\sigma^2_0}) \equiv P( {\cal M}_{\nu,\alpha_1,\sigma^2_1})$ on the paths of $Z(\ss), \ss \in D$, if and only if
\begin{equation}\label{condmat}
\sigma_0^2 / \alpha_0^{2\nu}=\sigma_1^2/ \alpha_1^{2\nu}.
\end{equation}
\end{theo}

The first relevant result of this paper concerns the characterization of the equivalence of two zero mean Gaussian measures under GW functions. The crux of the proof is the arguments in Equation~\eqref{spectralfinite2}, coupled with the asymptotic expansion of
the  spectral density as in Theorem~\ref{the3} and~\ref{the356}.

\begin{theo} \label{W_vs_W}
 For a given $\kappa \ge 0$, let $P( \varphiall[i])$,  $i=0, 1$, be two zero mean Gaussian  measures
and  let $\mu > \lambda+d/2$, with $\lambda$ as defined through Equation~\eqref{lawea}.
 For any
bounded infinite set $D\subset \R^d$, $d=1, 2, 3$,
$P( \varphiall[0]) \equiv P(\varphiall[1])$ on the paths of $Z(\ss), \ss \in D$ if and only if
\begin{equation} \label{condition1_iff}
\sigma_0^2/ \beta_0^{2 \kappa+1} = \sigma_1^2/ \beta_1^{2 \kappa+1}.
\end{equation} \end{theo}

\begin{proof}
We  first consider  the case $\kappa>0$.
Let us start with the sufficient part of the assertion.
From  Theorem~\ref{the3} (Point~\ref{3.3}), there exist two positive constants  $c_i$ and $C_i$ such that

\begin{equation*}
c_i\leq z^{2\lambda}
\hatvarphiall[i](z)
\leq C_i,\quad i=0,1.
 \end{equation*}
 In order to prove the sufficient part, we need to find conditions such that,  for some positive and finite $c$,
\begin{equation}\label{eq:6}
\int_{c}^{\infty}z^{d-1} \bigg(
\frac{\hatvarphiall[1](z)-\hatvarphiall[0](z)}{\hatvarphiall[0](z)}
 \biggr)^{2} {\rm d} z<\infty.
 \end{equation}
 
We proceed by direct construction, and, using Theorem~\ref{the3} (Points~\ref{3.1} and~\ref{3.2}), we find that, as $z\to\infty$,
\begin{align*}
\Bigg |&\frac{\hatvarphiall[1](z)-\hatvarphiall[0](z)}{\hatvarphiall[0](z)}\Bigg | \leq L^{\boldsymbol{\varsigma}}{c_0^{-1}} z^{2\lambda}
\Bigg |\sigma_{1}^{2}\beta_{1}^{d}\Big [c_{3}^{\boldsymbol{\varsigma}}(\beta_{1}z)^{-2\lambda}\big\{1+\mathcal{O}(z^{-2})\big\}\\
&\quad + c_{4}^{\boldsymbol{\varsigma}}(z\beta_{1})^{-(\mu+\lambda)}\big\{\cos(\beta_{1}z-c_{5}^{\boldsymbol{\varsigma}})+\mathcal{O}(z^{-1})\big\}\Big ]\\
&\quad - \sigma_{0}^{2}\beta_{0}^{d}\Big [c_{3}^{\boldsymbol{\varsigma}}(\beta_{0}z)^{-2 \lambda}\big\{1+\mathcal{O}(z^{-2})\big\} \\
&\quad + c_{4}^{\boldsymbol{\varsigma}}(z\beta_{0})^{-(\mu+\lambda)}\big\{\cos(\beta_{0}z-c_{5}^{\boldsymbol{\varsigma}})+\mathcal{O}(z^{-1})\big\}\Big ]\Bigg |\\
&\leq L^{\boldsymbol{\varsigma} }{c_0^{-1}}\Bigg|c_{3}^{\boldsymbol{\varsigma}}\Big[\sigma_{1}^{2}\beta_{1}^{-(1+2\kappa)}-\sigma_{0}^{2}\beta_{0}^{-(1+2\kappa)}\Big]+\mathcal{O}(z^{-2})\\
&\quad +c_{4}^{\boldsymbol{\varsigma}}z^{\lambda-\mu}\Big[\sigma_{1}^{2}\beta_{1}^{\tlambda}\cos(\beta_{1}z-c_{5}^{\boldsymbol{\varsigma}})
-\sigma_{0}^{2}\beta_{0}^{\tlambda}\cos(\beta_{0}z-c_{5}^{\boldsymbol{\varsigma}})\Big]\\
&\quad +c_{4}^{\boldsymbol{\varsigma}}z^{\lambda-\mu}\mathcal{O}(z^{-1})\big\{\sigma_{1}^{2}\beta_{1}^{\tlambda}-\sigma_{0}^{2}\beta_{0}^{\tlambda}\big\}\Bigg|,
\end{align*}
where  $\tlambda=d-(\mu+\lambda)$. 
Let us now write
\begin{align*}
A(z)&=c_{3}^{\boldsymbol{\varsigma}}[\sigma_{1}^{2}\beta_{1}^{-(1+2\kappa)}-\sigma_{0}^{2}\beta_{0}^{-(1+2\kappa)}]+\mathcal{O}(z^{-2}),\\
B(z)&=c_{4}^{\boldsymbol{\varsigma}}z^{\lambda-\mu}[\sigma_{1}^{2}\beta_{1}^{\tlambda}\cos(\beta_{1}z-c_{5}^{\boldsymbol{\varsigma}})-\sigma_{0}^{2}\beta_{0}^{\tlambda}\cos(\beta_{0}z-c_{5}^{\boldsymbol{\varsigma}})], \text{ and }\\
D(z)&=c_{4}^{\boldsymbol{\varsigma}}z^{\lambda-\mu}\mathcal{O}(z^{-1})\big\{\sigma_{1}^{2}\beta_{1}^{\tlambda}-\sigma_{0}^{2}\beta_{0}^{\tlambda}\big\}.
\end{align*}
 Then, a sufficient condition  for (\ref{eq:6}) is the following:
\begin{equation}\label{eq:6_1}
(L^{\boldsymbol{\varsigma}}/c_0)^2 \int_{c}^{\infty}z^{d-1}\big(A(z)+B(z)+D(z)\big)^{2}{\rm d} z<\infty.
 \end{equation}
Note  that $A(z)$ is of order $\mathcal{O}(z^{-2})$ under Condition~\eqref{condition1_iff}. 
 We claim that~\eqref{eq:6_1} is satisfied if  $\sigma_{1}^{2}\beta_{1}^{-(1+2\kappa)}=\sigma_{0}^{2}\beta_{0}^{-(1+2\kappa)}$ for  $\mu>\lambda+d/2$,  $d=1,2,3$. 

In fact,  we have, for $z\to\infty$,
\begin{equation*}
\begin{aligned}
 |B(z)|&\leq c_{4}^{\boldsymbol{\varsigma}}z^{\lambda-\mu}[\sigma_{1}^{2}\beta_{1}^{\tlambda}+\sigma_{0}^{2}\beta_{0}^{\tlambda}]\leq c_{6}z^{\lambda-\mu}, \qquad
\end{aligned}
\end{equation*}
and
\begin{equation*}
\begin{aligned}
 |D(z)|&\leq c_{4}^{\boldsymbol{\varsigma}}z^{\lambda-\mu}\mathcal{O}(z^{-1})\big\{\sigma_{1}^{2}\beta_{1}^{\tlambda}+\sigma_{0}^{2}\beta_{0}^{\tlambda}\big\}\\
 &\leq c_{7}c_{4}^{\boldsymbol{\varsigma}}z^{\lambda-\mu-1}\big\{\sigma_{1}^{2}\beta_{1}^{\tlambda}+\sigma_{0}^{2}\beta_{0}^{\tlambda}\big\}
 \leq c_{8}z^{\lambda-\mu-1}
\end{aligned}
\end{equation*}
with $c_6$, $c_7$ and $c_8$  being positive and finite constants. 
Expanding (\ref{eq:6_1})
we notice that the dominant terms are $A^2$ and $B^2$, independently on the cross products. These are respectively of the order ${\cal O}(z^{-4})$ 
and ${\cal O}(z^{2(\lambda-\mu)})$. This in turn implies that the integral~\eqref{eq:6_1} is finite   if $\sigma_{1}^{2}\beta_{1}^{-(1+2\kappa)}=\sigma_{0}^{2}\beta_{0}^{-(1+2\kappa)}$, for $\mu>\lambda+d/2$ and $d=1,2,3$ and
this implies that ~\eqref{eq:6} is satisfied under the same conditions. 
The sufficient part of our claim is thus proved.

The necessary   part follows the proof of \cite{Zhang:2004}.  For $\mu>\lambda+d/2$ and $d=1,2,3$, we suppose $\sigma_{1}^{2}\beta_{1}^{-(1+2\kappa)}\neq\sigma_{0}^{2}\beta_{0}^{-(1+2\kappa)}$
  and let $\sigma_{2}^{2}=\sigma_{0}^{2}(\beta_{0}/\beta_{1})^{-(1+2\kappa)}$.
Then $\varphiall[0]$ and $\varphiall[1][2]$ define two equivalent Gaussian measures. We need to show  that
$\varphiall[1][2]$ and $\varphiall[1]$ define two orthogonal Gaussian measures.
The rest of the proof follows the same arguments in \cite{Zhang:2004}.

We omit the proof of the special case  $\kappa=0$, since is similar to the case $\kappa>0$,  but using the arguments in Theorem~\ref{the356}.
\end{proof}

An immediate consequence of Theorem~\ref{W_vs_W} is that for fixed $\kappa$ and $\mu$, the $\beta$ and $\sigma^2$ parameters cannot be estimated consistently  \citep{Zhang:2004}. Instead, the microergodic parameter $\sigma^{2}\beta^{-(1+2\kappa)}$
is consistently estimable. In Section~\ref{4}, we establish the asymptotic properties of ML estimation associated to the microergodic parameter.

The next result depicts an interesting scenario in which a GW and  Mat{\'e}rn model are considered
and gives sufficient conditions for the compatibility of these two covariance models.
We offer a constructive proof, the crux of the argument being  again  Equation~\eqref{spectralfinite2}.
We treat  the cases $\kappa>0$ and $\kappa=0$ separately.

\begin{theo}\label{ThmX}
For given $\nu>1/2$ and $\kappa>0$,   let $P({\cal M}_{\nu,\alpha,\sigma_0^2})$ and $P( \varphiall[][1])$
  be two zero mean Gaussian measures. If $\nu=\kappa+1/2$, $\mu > \lambda+d/2$, with $\lambda$ as defined through Equation~\eqref{lawea}, and
\begin{equation}\label{cafu}
\sigma_{0}^{2}\alpha^{-2\nu}=C_{\nu,\kappa,\mu}\sigma_1^2\beta^{-(1+2\kappa)},
\end{equation}. 
 where $C_{\nu,\kappa,\mu}= \frac{ \mu2^{-d}\Gamma(\nu)\Gamma(\kappa)\Gamma(2\kappa+d)}{\Gamma(\nu+d/2)\Gamma(\kappa+d/2)B(2\kappa,\mu+1)}$,
then for any
bounded infinite set $D\subset \R^d$, $d=1, 2, 3$, $P({\cal M}_{\nu,\alpha,\sigma^2_0})
  \equiv P(\varphiall[][1])$ on the paths of $Z(\ss), \ss \in D$.
\end{theo}

\begin{proof}

In order to prove Theorem 5, we need to find conditions such that  for some positive and finite $c$,
\begin{equation}\label{eq:99}
\int_{c}^{\infty}z^{d-1} \bigg(
\frac{\widehat{\varphi}_{\nu,\kappa,\beta,\sigma^2_1}(z)-\widehat{{\cal M}}_{\nu,\alpha,\sigma^2_0}(z)}{\widehat{{\cal M}}_{\nu,\alpha,\sigma^2_0}(z)}
 \biggr)^{2} {\rm d} z<\infty.
\end{equation}

It is known that $\widehat{{\cal M}}_{\nu,\alpha,\sigma^2_0}(z)z^a$ is bounded away from 0 and $\infty$ as $z \to \infty$
for some $a>0$ \citep{Zhang:2004}.
Using~\eqref{stein1} and  Theorem~\ref{the3} (Points~\ref{3.1} and~\ref{3.2}), we have, as $z\to\infty$,
\begin{equation*}
\begin{aligned}
\Bigg|&\frac{\hatvarphiall[][1](z)-\widehat{{\cal M}}_{\nu,\alpha,\sigma^2_0}(z)}{\widehat{{\cal M}}_{\nu,\alpha,\sigma^2_0}(z)}\Bigg|\\
&= \Bigg|\frac{\sigma_{1}^{2}\beta^{d}\Gamma(\nu)L^{\boldsymbol{\varsigma}}}{\Gamma(\nu+d/2)\sigma_{0}^{2}\alpha^{-2\nu}\pi^{-\frac{d}{2}}}\Big[c_{3}^{\boldsymbol{\varsigma}}(\beta z)^{-2\lambda}\big\{1+\mathcal{O}(z^{-2})\big\}\\
&\quad +c_{4}^{\boldsymbol{\varsigma}}(z\beta)^{-(\mu+\lambda))}\big\{\cos(\beta z-c_{5}^{\boldsymbol{\varsigma}})+\mathcal{O}(z^{-1})\big\}\Big](\alpha^{-2}+z^{2})^{\nu+\frac{d}{2}}-1\Bigg|\\
&=\Bigg|\frac{\sigma_{1}^{2}\beta^{d}\Gamma(\nu)L^{\boldsymbol{\varsigma}}}{\Gamma(\nu+d/2)\sigma_{0}^{2}\alpha^{-2\nu}\pi^{-\frac{d}{2}}}\Big[c_{3}^{\boldsymbol{\varsigma}}(\beta z)^{-2\lambda}\big\{1+\mathcal{O}(z^{-2})\big\}\\
&\quad+ c_{4}^{\boldsymbol{\varsigma}}(z\beta)^{-(\mu+\lambda)}\big\{\cos(\beta z-c_{5}^{\boldsymbol{\varsigma}})+\mathcal{O}(z^{-1})\big\}\Big]z^{2\nu+d}(({\alpha z})^{-2}+1)^{\nu+\frac{d}{2}}-1\Bigg|\\
&=\Bigg| w_{1}  z^{-2\lambda}\big\{1+\mathcal{O}(z^{-2})\big\}z^{2\nu+d} \Big[1 + (\nu+d/2)(\alpha z)^{-2}+\mathcal{O}(z^{-2})\Big]\\&\quad+
 w_{2}   z^{-(\mu+\lambda)}z^{2\nu+d} \Big[1+(\nu+d/2)(\alpha z)^{-2}+\mathcal{O}(z^{-2})\Big]
\big\{\cos(\beta z-c_{5}^{\boldsymbol{\varsigma}})+\mathcal{O}(z^{-1})\big\}-1\Bigg|,
\end{aligned}
\end{equation*}
where   $w_{1}=\frac{L^{\boldsymbol{\varsigma}}\sigma_{1}^{2}\beta^{-( 1+2\kappa)}\Gamma(\nu)c_{3}^{\boldsymbol{\varsigma}}}{\Gamma(\nu+d/2)\sigma_{0}^{2}\alpha^{-2\nu}\pi^{-{d}/{2}}}$,
$w_{2}=w_{1}c_{4}^{\boldsymbol{\varsigma}}\beta^{\lambda-\mu}/c_{3}^{\boldsymbol{\varsigma}}$.
Since $z^{2\nu+d}\big[(\nu+d/2)(\alpha z)^{-2}+\mathcal{O}(z^{-2})\big]=\mathcal{O}(z^{2\nu+d-2})$, we have
\begin{align*}
&\int_{c}^{\infty}z^{d-1}\Bigg|\frac{\widehat{\varphi}_{\mu,\kappa,\beta,\sigma^2_1}(z)-\widehat{{\cal M}}_{\nu,\alpha,\sigma^2_0}(z)}{\widehat{{\cal M}}_{\nu,\alpha,\sigma^2_0}(z)}\Bigg|^2 {\rm d} z \\
 &= \int_c^{\infty}z^{d-1}\Bigg|w_{1}z^{-2\lambda}\mathcal{O}(z^{2\nu+d-2})+\big\{w_{1}z^{2\nu-( 1+2\kappa)}-1\big\}+w_{1}z^{-2\lambda}\\
 &\quad\times \big\{ \mathcal{O}(z^{2\nu+d-2})+ \mathcal{O}(z^{2\nu+d-4})\big\}+w_{2}z^{-(\mu+\lambda)}\big\{\mathcal{O}(z^{2\nu+d-2})+z^{2\nu+d}\big\}\big\{\cos(\beta z-c_{5}^{\boldsymbol{\varsigma}})\\
 &\quad+\mathcal{O}(z^{-1})\big\}\Bigg|^2 {\rm d }z.
\end{align*}
For assessing the last integral, the following is relevant:
\begin{itemize}
  \item[(i)] $w_{1}z^{2\nu-( 1+2\kappa)}-1=0$ if $\nu=\kappa+1/2$ and  $w_1=1$.
  \item[(ii)] $\int_{c}^{\infty}z^{d-1}\big(w_{1}z^{-2\lambda}\mathcal{O}(z^{2\nu+d-2})\big)^2 {\rm d} z<\infty$ if $d=1,2,3$ and $\nu=\kappa+1/2$.
  \item[(iii)] $\int_{c}^{\infty}z^{d-1}\Big(w_{1}z^{-2\lambda}\big\{ \mathcal{O}(z^{2\nu+d-2})+ \mathcal{O}(z^{2\nu+d-4})\big\}\Big)^2  {\rm d} z<\infty$ if $d=1,2,3$ and $\nu=\kappa+1/2$.
  \item[(iv)] $\int_{c}^{\infty}z^{d-1}\Big(w_{2}z^{-(\mu+\lambda)}\big\{\mathcal{O}(z^{2\nu+d-2})+z^{2\nu+d}\big\}\big\{\cos(\beta z-c_{5}^{\boldsymbol{\varsigma}})+\mathcal{O}(z^{-1})\big\}\Big)^2  {\rm d} z<\infty$ if $\mu>\lambda+d/2$ and $\nu=\kappa+1/2$.
  \item[(v)] $\int_{c}^{\infty}z^{d-1}\big(w_{1}z^{-2\lambda}\mathcal{O}(z^{2\nu+d-2})\big)\big(w_{1}z^{-2\lambda}\mathcal{O}(z^{-2})(\mathcal{O}(z^{2\nu+d-2})+z^{2\nu+d})\big)  {\rm d} z<\infty$ if $d=1,2,3$ and $\nu=\kappa+1/2$.
  \item[(vi)] $\int_{c}^{\infty}z^{d-1}\big(w_{1}z^{-2\lambda}\mathcal{O}(z^{2\nu+d-2})\big)\big(w_{2}z^{-(\mu+\lambda)}\big\{\mathcal{O}(z^{2\nu+d-2})+z^{2\nu+d}\big\}\big\{\cos(\beta z-c_{5}^{\boldsymbol{\varsigma}})+\mathcal{O}(z^{-1})\big\}\big)  {\rm d} z<\infty$ if $\mu>\lambda+d-2$ and $\nu=\kappa+1/2$.
   \item[(vii)]   $\int_{c}^{\infty}z^{d-1}\big(w_{1}z^{-2\lambda}\mathcal{O}(z^{-2})\big\{\mathcal{O}(z^{2\nu+d-2})+z^{2\nu+d}\big\}\big)\big(w_{2}z^{-(\mu+\lambda)}\big\{\mathcal{O}(z^{2\nu+d-2})+z^{2\nu+d}\big\}$$\times$$\big\{\cos(\beta z-c_{5}^{\boldsymbol{\varsigma}})+\mathcal{O}(z^{-1})\big\}\big)   {\rm d} z<\infty$ if $\mu>\lambda+d-2$ and $\nu=\kappa+1/2$.
\end{itemize}
This allows us to conclude that, for a given $\kappa>0$, if $w_1=1$, $\nu=\kappa+1/2$,  $\mu>\lambda+d/2$  and $d=1,2,3$, then~\eqref{eq:99} holds  and thus $P({\cal M}_{\nu,\alpha,\sigma^2_0}) \equiv  P(\varphiall[][1])$.

Condition $w_1=1$ is equivalent to
\begin{align*}
L^{\boldsymbol{\varsigma}}c_{3}^{\boldsymbol{\varsigma}}\sigma_{1}^{2}\beta^{-(1+2\kappa)}=\pi^{-d/2}\Gamma(\nu+d/2)\Gamma(\nu)^{-1}\sigma_{0}^{2}\alpha^{-2\nu},
\end{align*}
and
from the  definition of   $c_{3}^{\boldsymbol{\varsigma}}$ and $L^{\boldsymbol{\varsigma}}$,  the previous  condition can be 
rewritten as
  $\sigma_{0}^{2}\alpha^{-2\nu}=C_{\nu,\kappa,\mu}\sigma_{1}^{2}\beta^{-(1+2\kappa)}$.
%
\end{proof}
\begin{theo}\label{ThmX0}
Let $P({\cal M}_{1/2,\alpha,\sigma_0^2})$ and $P(\varphi_{\mu,0,\beta,\sigma^2_1})$
  be two zero mean Gaussian measures. If  $\mu > d+1/2$ and
\begin{equation}\label{coutinho}
\sigma_{0}^{2}\alpha^{-2\nu}=R_{\mu}\sigma_{1}^{2}\beta^{-1},
\end{equation}
 where $R_{\mu}= \frac{\mu 2^{1-d}\Gamma(1/2)\Gamma(d)}{\Gamma(1/2+d/2)\Gamma(d/2)}$,
then for any
bounded infinite set $D\subset \R^d$, $d=1, 2, 3$, $P({\cal M}_{1/2,\alpha,\sigma^2_0})
  \equiv P(\varphi_{\mu,0,\beta,\sigma^2_1})$ on the paths of $Z(\ss), \ss \in D$.
\end{theo}
\begin{proof}
The proof follows the same arguments exposed for the case $\kappa>0$ in Theorem 5 , but using~\eqref{stein1} and  Theorem~\ref{the356} (Points 1 and 2).
In this case, it can be shown that if  $\mu>d+1/2$, $d=1,2,3$ and $\left(\frac{\mu 2^{1-d}\Gamma(1/2)\Gamma(d)}{\Gamma(1/2+d/2)\Gamma(d/2)}\right)\sigma_{1}^{2}\beta^{-1}=\sigma_{0}^{2}\alpha^{-2\nu}$
   then~\eqref{eq:99} holds.
\end{proof}

{\bf Remark}:
In Theorems 5  and 6 since $\nu=\kappa+1/2$  for $\kappa\geq0$, using the duplication formula of the gamma function,
we easily obtain  $C_{\kappa+1/2,\kappa,\mu}=\mu {\Gamma(2\kappa+\mu+1)}/{\Gamma(\mu+1)}$,
and $R_{\mu}=\mu$ in (\ref{cafu}) and  (\ref{coutinho})   respectively.
Then the  sufficient condition for $P({\cal M}_{\nu,\alpha,\sigma^2_0}) \equiv  P(\varphiall[][1])$ can  be simplified as :
\begin{equation}\label{true}
\sigma_{0}^{2}\alpha^{-2\nu}=\left( \frac{\mu\Gamma(2\kappa+\mu+1)}{\Gamma(\mu+1)} \right) \sigma_{1}^{2}\beta^{-(1+2\kappa)},
\end{equation}
$ \nu=\kappa+1/2$, $\mu > \lambda+d/2$ and $d=1,2,3$ for  $\kappa\geq0$.

\section{Asymptotic properties of the ML estimation for the GW model}\label{4}

We now focus on the microergodic parameter  $\sigma^{2}\beta^{-( 1+2\kappa)}$ associated to  the GW family. The following results fix the  asymptotic properties of its ML estimator. In particular, we will show that the microergodic parameter can be estimated consistently, and then we will assess the asymptotic distribution of the ML estimator.

Let $D\subset \R^d$ be  a bounded subset of $ \R^d$ 
and  $S_n=\{ \ss_1,\ldots,\ss_n \in D \subset \R^d \}$
 denote any set of distinct locations.
Let $\bZ_n=(Z(\boldsymbol{s}_1),\ldots,Z(\boldsymbol{s}_n))^{\prime}$
be a finite  realization of   a zero mean stationary Gaussian field with  a given parametric covariance function
$\sigma^2 \phi(\cdot; \btau)$, with $\sigma^2>0$, $\btau$ a parameter vector and  $\phi$ a member of the class $\Phi_d$, with $\phi(0; \btau)=1$.

We then write 
$R_{n}(\btau)=[\phi(\|\boldsymbol{s}_i-\boldsymbol{s}_j\|; \btau)]_{i,j=1}^n$ for the associated correlation matrix.
The  Gaussian log-likelihood function is defined as:
\begin{equation}\label{eq:17}
\mathcal{L}_{n}(\sigma^{2},\btau)=-\frac{1}{2} \left(n\log(2\pi\sigma^{2})+\log(|R_{n}(\btau)|)+\frac{1}{\sigma^{2}}\bZ_n^{\prime}R_{n}(\btau)^{-1}\bZ_n \right).
\end{equation}
Under the Mat{\'e}rn model, the Gaussian log-likelihood is obtained with $\phi(\cdot; \btau)\equiv {\cal M}_{\nu,\alpha,1}$
and $\btau=(\nu,\alpha)^{\prime}$.
Since in what follows  $\nu$ is assumed known and  fixed, for notation convenience, we write $\tau=\alpha$.
Let $\hat{\sigma}^2_n$ and  $\hat{\alpha}_n$ be the maximum likelihood estimator obtained maximizing
$\mathcal{L}_{n}(\sigma^{2},\alpha)$ for a fixed $\nu$.

Below, we report a result that
establishes strong consistency and asymptotic distribution of the ML estimation of the microergodic parameter of the Mat{\'e}rn  model.

\begin{theo}[\citealp{Shaby:Kaufmann:2013}] \label{theo8}
Let $Z(\boldsymbol{s})$, $\boldsymbol{s}\in D\subset \R^d$, $d=1,2,3$, be a zero mean Gaussian field  with a Mat{\'e}rn  covariance model ${\cal M}_{\nu,\alpha_0,\sigma^2_0}$.
  Suppose $(\sigma_{0}^{2},\alpha_{0})^{\prime} \in (0,\infty)\times [\alpha_{L},\alpha_{U}]$, for any $0<\alpha_{L}<\alpha_{U}<\infty$. Let $(\hat{\sigma}_{n}^{2},\hat{\alpha}_{n})^{\prime}$ maximize~\eqref{eq:17} over $(0,\infty)\times[\alpha_{L},\alpha_{U}]$. Then as $n\to\infty$,
\begin{enumerate}
\item $\hat{\sigma}_{n}^{2}/\hat{\alpha}_{n}^{2\nu}\stackrel{a.s.}{\longrightarrow} \sigma_{0}^{2}/\alpha_{0}^{2\nu}$, and
\item $\sqrt{n}(\hat{\sigma}_{n}^{2}/\hat{\alpha}_{n}^{2\nu}-\sigma_{0}^{2}/\alpha_{0}^{2\nu})\stackrel{\mathcal{D}}{\longrightarrow} \mathcal{N}(0,2(\sigma_{0}^{2}/\alpha_{0}^{2\nu})^2)$.
\end{enumerate}
\end{theo}
Analogous results can be found in \citep{Zhang:2004,Wang:Loh:2011},
when $\hat{\alpha}_{n}$ is replaced by $\alpha$,  an arbitrary  positive fixed constant. \cite{Shaby:Kaufmann:2013} show, through simulation study,
that
asymptotic  approximation using a fixed scale parameter can be problematic when applied
to finite samples, even for large sample sizes. In contrast, they show  that performance
is improved and asymptotic approximations are applicable for smaller sample sizes, when
the parameters are jointly estimated.

Now, let us consider the Gaussian log-likelihood under the GW model, so that $\btau=(\mu,\kappa,\beta)^{\prime}$ and $\phi(\cdot;\btau)= \varphi_{\mu,\kappa,\beta,1}(\cdot)$ according to the previous notation. Since in what follows $\kappa$ and $\mu$ are assumed known and  fixed, for notation convenience we write $\tau=\beta$.
To prove the analogue of Theorem~\ref{theo8}  for the GW case, we consider
two types of estimators. The first maximizes~\eqref{eq:17} with respect to $ \sigma^{2}$ for  a fixed arbitrary compact support $\beta>0$,
obtaining the following estimator
 \begin{equation} \label{hoceini}
 \hat{\sigma}_n^2(\beta)=\argmax_{\sigma^2} \mathcal{L}_{n}(\sigma^{2},\beta)=\bZ_n^{\prime}R_{n}(\beta)^{-1}\bZ_n/n. \end{equation}
  Here $R_{n}(\beta)$ is the  correlation matrix  coming from the GW family $\varphi_{\mu,\kappa,\beta,1}$.
 The following result offers some asymptotic properties of the sequence of   random variables
  $\hat{\sigma}_{n}^{2}(\beta)/\beta^{( 1+2\kappa)}$.


\begin{theo}\label{theo10}
 Let $Z(\boldsymbol{s})$, $\boldsymbol{s}\in D\subset \R^d$, $d=1,2,3$,
 be a zero mean Gaussian field  with  GW covariance model
 $\varphiall[0]$, with  $\mu>\lambda+d/2$.
  Suppose $(\sigma_{0}^{2},\beta_0)\in (0,\infty)\times  (0,\infty)$. For a fixed $\beta>0$, let $\hat{\sigma}_{n}^{2}(\beta)$ as defined through Equation~\eqref{hoceini}. Then,  as $n\to\infty$,
\begin{enumerate}
\item $\hat{\sigma}_{n}^{2}(\beta)/\beta^{ 1+2\kappa} \stackrel{a.s}{\longrightarrow}\sigma_0^{2}(\beta_0)/\beta_0^{ 1+2\kappa}$ and
\item $\sqrt{n}(\hat{\sigma}_{n}^{2}(\beta)/\beta^{ 1+2\kappa}-\sigma_0^{2}(\beta_0)/\beta_0^{ 1+2\kappa})\stackrel{\mathcal{D}}{\longrightarrow}\mathcal{N}(0,2(\sigma_0^{2}(\beta_0)/\beta_0^{ 1+2\kappa})^{2})$.
\end{enumerate}
\end{theo}
\begin{proof}
The proof of the first assertion follows the same arguments of the
  proof of Theorem 3 in \cite{Zhang:2004}, and we omit it.
  The proof of the second assertion is quite technical  and long  and it has been deferred to the supplementary material.
\end{proof}

The second type of estimation technique considers the  joint maximization of~\eqref{eq:17}  with respect
to
$(\sigma^{2},\beta)\in (0,\infty)\times I$, where $I=[\beta_L, \beta_U]$ and $0<\beta_L<\beta_U<\infty$.
The solution of this optimization problem is  given by
$(\hat{\sigma}_n^2(\hat{\beta}_n),\hat{\beta}_n)$ where
 $$\hat{\sigma}_n^2(\hat{\beta}_n)= \bZ_n^{\prime}R_{n}(\hat{\beta}_n)^{-1}\bZ_n/n$$
and $\hat{\beta}_n=\argmax_{\beta \in I} \mathcal{PL}_{n}(\beta)$.  Here $\mathcal{PL}_{n}(\beta)$
is the profile log-likelihood:
\begin{equation}\label{eq:prof}
\mathcal{PL}_{n}(\beta)=-\frac{1}{2} \left( \log(2\pi)+n\log(\hat{\sigma}_n^2(\beta))+\log|R_{n}(\beta)| +n \right).\end{equation}
In order to establish strong consistency and asymptotic distribution
 of the sequence of   random variables
  $\hat{\sigma}_{n}^{2}(\hat{\beta}_{n})/\hat{\beta}_{n}^{ 1+2\kappa}$,  we use the following Lemma that establishes
the monotone behaviour of   $\hat{\sigma}_{n}^{2}(\beta)/\beta^{ 1+2\kappa}$ when viewed as a function of $\beta \in I$
under specific condition on the $\mu$ parameter.


\begin{lemma}\label{lemmaML}
 For any $ \beta_1< \beta_2 \in I$ and for each $n$, 
 $\hat{\sigma}_{n}^{2}(\beta_1)/\beta_1^{ 1+2\kappa} \leq \hat{\sigma}_{n}^{2}(\beta_2)/\beta_2^{ 1+2\kappa} $
if and only if $ \mu\geq \lambda+3$.
\end{lemma}

\begin{proof}
The proof follows  \cite{Shaby:Kaufmann:2013}.
Let  $0<\beta_1<\beta_2$, with $ \beta_1, \beta_2 \in I$. Then, for any $\bZ_n$,
$$\hat{\sigma}_{n}^{2}(\beta_1)/\beta_1^{ 1+2\kappa}- \hat{\sigma}_{n}^{2}(\beta_2)/\beta_2^{ 1+2\kappa}=\frac{1}{n}
\bZ_n^{\prime}(R_{n}(\beta_1)^{-1}\beta_1^{-( 1+2\kappa)} -R_n(\beta_2)^{-1}\beta_2^{-( 1+2\kappa)})\bZ_n $$
is nonnegative if the matrix  $R_{n}(\beta_1)^{-1}\beta_1^{-( 1+2\kappa)} -R_n(\beta_2)^{-1}\beta_2^{-( 1+2\kappa)}$  is positive semidefinite and this happens if and only if  the matrix
  $B=R_{n}(\beta_2)\beta_2^{ 1+2\kappa} -R_n(\beta_1)\beta_1^{ 1+2\kappa}$
  with generic element
    $$B_{ij}=\beta_2^{ 1+2\kappa}\varphi_{\mu,\kappa,\beta_2,1}(||\ss_i-\ss_j||)      -  \beta_1^{ 1+2\kappa}\varphi_{\mu,\kappa,\beta_1,1}(||\ss_i-\ss_j||). $$
  is  positive semidefinite.
 From Theorem 2 in \cite{Porcu:Zastavnyi:Xesbaiat}, this happens  if and only if  $\mu\geq \lambda+3$.
\end{proof}

\begin{theo}\label{theo11}
Let $Z(\boldsymbol{s})$, $\boldsymbol{s}\in D\subset \R^d$, $d=1, 2, 3$, be a zero mean Gaussian field
with a GW covariance model
$\varphiall[0]$ with $ \mu\geq \lambda+3$.
 Suppose $(\sigma_{0}^{2},\beta_{0})\in (0,\infty)\times
I$ where  $I=[\beta_L, \beta_U]$ with $0<\beta_L<\beta_U<\infty$. Let
$(\hat{\sigma}_{n}^{2},\hat{\beta}_{n})^{\prime}$ maximize~\eqref{eq:17} over $(0,\infty)\times I$. Then as
$n\to\infty$,
\begin{enumerate}
\item $\hat{\sigma}_{n}^{2}(\hat{\beta}_{n})/\hat{\beta}_{n}^{ 1+2\kappa}   \stackrel{a.s}{\longrightarrow}\sigma_0^{2}(\beta_0)/\beta_0^{1+2\kappa}$ and
\item  $\sqrt{n}(\hat{\sigma}_{n}^{2}(\hat{\beta}_{n})/\hat{\beta}_{n}^{ 1+2\kappa}  -\sigma_0^{2}(\beta_0)/\beta_0^{ 1+2\kappa})\stackrel{\mathcal{D}}{\longrightarrow} \mathcal{N}(0,2(\sigma_0^{2}(\beta_0)/\beta_0^{ 1+2\kappa})^{2})$.
\end{enumerate}
\end{theo}

\begin{proof}
The proof follows  \cite{Shaby:Kaufmann:2013} which uses the same arguments in the Mat{\'e}rn case.
Let ${\cal G}_n(x)=\hat{\sigma}_{n}^{2}(x)/x^{ 1+2\kappa}  $ and
 define  the sequences ${\cal G}_n(\beta_L)$      and ${\cal G}_n(\beta_U)$.
Since $\beta_{L}\leq \hat{\beta}_{n} \leq \beta_{U}$  for every $n$, then, using Lemma  \ref{lemmaML},
${\cal G}_n(\beta_U) \leq {\cal G}_n(\hat{\beta}_{n})    \leq {\cal G}_n(\beta_L)$
for all $n$ with probability
one. Combining this with Theorem \ref{theo10}  implies the result.
\end{proof}

\section{Prediction using GW model}\label{5}
We now consider  prediction of a Gaussian field at a new location $\ss_0$,
using the GW model, under fixed domain asymptotics.
Specifically, we focus on
two properties: asymptotic efficiency  prediction and asymptotically correct
estimation of prediction variance.
\citet{Stein:1988} shows that both asymptotic
properties hold when the  Gaussian measures are equivalent.
Let $P(\varphiall[i])$, $i=1,2$, be two probability zero mean Gaussian measures. Under
$P( \varphiall[0])$, and using Theorem~\ref{W_vs_W},  both properties hold when
$\sigma_0^{2}\beta_0^{-(1+2\kappa)}=\sigma_1^{2}\beta_1^{-(1+2\kappa)}$, $\mu>\lambda+d/2$ and $d=1,2,3$.

Similarly, let $P( {\cal M}_{\nu,\alpha,\sigma_2^2})$  and $P( \varphiall[1])$
be two Gaussian measures with Mat{\'e}rn and  GW model.
Under
 $P({\cal M}_{\nu,\alpha,\sigma_2^2})$      both properties hold when (\ref{true}) is true, $\mu>\lambda+d/2$,  $d=1,2,3$.
 Actually, \cite{Stein:1993} gives a substantially weaker condition for asymptotic efficiency  prediction based on
the asymptotic behaviour of the ratio of the isotropic spectral densities. Now, let
\begin{equation}\label{blup}
\widehat{Z}_{n}(\mu,\kappa,\beta)=\c_n(\mu,\kappa,\beta)^{\prime}R_{n}(\mu,\kappa,\beta)^{-1}\bZ_n
\end{equation}
be the best linear unbiased predictor at an unknown location $\ss_0\in D\subset \R^d$,
under the misspecified model $P( \varphiall)$,
where $\c_n(\mu,\kappa,\beta)=[\varphi_{\mu,\kappa,\beta,1}(\|\boldsymbol{s}_0-\boldsymbol{s}_i \|)]_{i=1}^n$
 and $R_{n}(\mu,\kappa,\beta)=[\varphi_{\mu,\kappa,\beta,1}(\|\boldsymbol{s}_i-\boldsymbol{s}_j \|) ]_{i,j=1}^n$
 is the correlation matrix.

If the correct model is $P( \varphiall[0])$, then the mean squared error of the predictor is given by:
\begin{align}\label{mse_miss}
&\var_{\mu,\kappa,\beta_0,\sigma^2_0}\left[\widehat{Z}_{n}(\mu,\kappa,\beta)-Z(\boldsymbol{s}_0)\right]=\sigma_0^2\Big(1-2\c_n(\mu,\kappa,\beta)^{\prime}R_{n}(\mu,\kappa,\beta)^{-1}\c_n(\mu,\kappa,\beta_0)\\ &\quad+ \c_n(\mu,\kappa,\beta)^{\prime}R_{n}(\mu,\kappa,\beta)^{-1} R_{n}(\mu,\kappa,\beta_0) R_{n}(\mu,\kappa,\beta)^{-1}\c_n(\mu,\kappa,\beta)\Big)\nonumber.
\end{align}
In the case that
$\beta_0= \beta$, i.e., true and wrong models coincide, this expression simplifies to
\begin{align}\label{msetrue}
\var_{\mu,\kappa,\beta_0,\sigma^2_0}\big[&\widehat{Z}_{n}(\mu,\kappa,\beta_0)-Z(\boldsymbol{s}_0)\big]\\
 \nonumber   &=\sigma_0^2\big(1-\c_n(\mu,\kappa,\beta_0)^{\prime}R_{n}(\mu,\kappa,\beta_0)^{-1}\c_n(\mu,\kappa,\beta_0)\big).
\end{align}
Similarly  $\var_{\nu,\alpha,\sigma^2_2}\big[\widehat{Z}_{n}(\mu,\kappa,\beta)-Z(\boldsymbol{s}_0)\big]$ and $\var_{\nu,\alpha,\sigma^2_2}\big[\widehat{Z}_{n}(\nu,\alpha)-Z(\boldsymbol{s}_0)\big]$ can be defined  under $P({\cal M}_{\nu,\alpha,\sigma_2^2})$, where $\widehat{Z}_{n}(\nu,\alpha)$ is the best linear unbiased predictor using the Mat{\'e}rn model.
The following results are an application of Theorems~1 and~2 of \cite{Stein:1993}.

\begin{theo}\label{kauf_3}
Let $P(\varphiall[0])$, $P(\varphiall[1])$,
$P( {\cal M}_{\nu,\alpha,\sigma^2_2})$ be three Gaussian probability measures on $D\subset \R^d$
and let $\mu>\lambda$. Then, for all $\boldsymbol{s}_0\in D$:
\begin{enumerate}
  \item Under $P(\varphiall[0])$, as $n\to \infty$,
  \begin{equation}\label{kauf3_1} \frac{\var_{\mu,\kappa,\beta_0,\sigma^2_0}\bigl[\widehat{Z}_{n}(\mu,\kappa,\beta_1)-Z(\boldsymbol{s}_0)\bigr]}{\var_{\mu,\kappa,\beta_0,\sigma^2_0}\bigl[\widehat{Z}_{n}(\mu,\kappa,\beta_0)-Z(\boldsymbol{s}_0)\bigr]}{\,\longrightarrow\,}1,
        \end{equation}
        for any fixed $\beta_1>0$.
   \item Under $P( {\cal M}_{\nu,\alpha,\sigma^2_2})$,   if $\nu=\kappa+1/2$  as $n\to \infty$,
  \begin{equation}\label{kauf3_11} \frac{\var_{\nu,\alpha,\sigma^2_2}\bigl[\widehat{Z}_{n}(\mu,\kappa,\beta_1)-Z(\boldsymbol{s}_0)\bigr]}{\var_{\nu,\alpha,\sigma^2_2}\bigl[\widehat{Z}_{n}(\nu,\alpha)-Z(\boldsymbol{s}_0)\bigr]}{\,\longrightarrow\,}1,
        \end{equation}
           for any fixed $\beta_1>0$.
  \item Under $P( \varphiall[0])$,  if $\sigma_0^{2}\beta_0^{-(1+2\kappa)}=\sigma_1^{2}\beta_1^{-(1+2\kappa)}$, then as $n\to \infty$,
 \begin{equation}\label{kauf3_2} \frac{\var_{\mu,\kappa,\beta_1,\sigma^2_1}\bigl[\widehat{Z}_{n}(\mu,\kappa,\beta_1)-Z(\boldsymbol{s}_0)\bigr]}{\var_{\mu,\kappa,\beta_0,\sigma^2_0}\bigl[\widehat{Z}_{n}(\mu,\kappa,\beta_1)-Z(\boldsymbol{s}_0)\bigr]}{\,\longrightarrow\,}1.
 \end{equation}
 \item  Under $P( {\cal M}_{\nu,\alpha,\sigma^2_2})$, if
$\mu {\Gamma(2\kappa+\mu+1)}\big/{\Gamma(\mu+1)} \times \sigma_{1}^{2}\beta_1^{-(1+2\kappa)}=\sigma_{2}^{2}\alpha^{-2\nu}$, $\nu=\kappa+1/2$, then as $n\to \infty$,
 \begin{equation}\label{kauf3_3} \frac{\var_{\mu,\kappa,\beta_1,\sigma^2_1}\bigl[\widehat{Z}_{n}(\mu,\kappa,\beta_1)-Z(\boldsymbol{s}_0)\bigr]}{\var_{\nu,\alpha,\sigma^2_2}\bigl[\widehat{Z}_{n}(\mu,\kappa,\beta_1)-Z(\boldsymbol{s}_0)\bigr]}{\,\longrightarrow\,}1.
 \end{equation}
\end{enumerate}
\end{theo}

\begin{proof}
Since $\hatvarphiall[0](z)$ is bounded away from zero and infinity  and  as $z\to \infty$,
\begin{equation*}
\begin{aligned}
\frac{\hatvarphiall[1](z)}{\hatvarphiall[0](z)} 
&=\frac{\sigma_1^2\beta_1^{d}\Big[c_{3}^{\boldsymbol{\varsigma}}\beta_1^{-2\lambda}\big\{1+\mathcal{O}(z^{-2})\big\}
\displaystyle+\,c_{4}^{\boldsymbol{\varsigma}}\beta_1^{-(\mu+\lambda)}z^{\lambda-\mu}\big\{\cos(z\beta_1-c_{5}^{\boldsymbol{\varsigma}})+\mathcal{O}(z^{-1})\big\}\Big]}
{\sigma_0^2\beta_0^{d}\Big[c_{3}^{\boldsymbol{\varsigma}}\beta_0^{-2\lambda}\big\{1+\mathcal{O}(z^{-2})\big\}
\displaystyle+\,c_{4}^{\boldsymbol{\varsigma}}\beta_0^{-(\mu+\lambda)}z^{\lambda-\mu}\big\{\cos(z\beta_0-c_{5}^{\boldsymbol{\varsigma}})+\mathcal{O}(z^{-1})\big\}\Big]}
\end{aligned}
\end{equation*}
then, for $\mu>\lambda$,  we have
 \begin{equation}\label{lim2}
 \underset{z \to \infty}{\lim}\frac{\hatvarphiall[1](z)}{\hatvarphiall[0](z)}=\frac{\sigma_1^{2}\beta_1^{-(1+2\kappa)}}{\sigma_0^{2}\beta_0^{-(1+2\kappa)}}
\end{equation}
and, using Theorem~1 of \cite{Stein:1993},  we obtain~\eqref{kauf3_1}.
If $\sigma_1^{2}\beta_1^{-(1+2\kappa)}=\sigma_0^{2}\beta_0^{-(1+2\kappa)}$  and using Theorem~2 of \cite{Stein:1993},  we obtain~\eqref{kauf3_2}.

Similarly, since  $\widehat{{\cal M}}_{\nu,\alpha,\sigma^2_2}(z)$ is bounded away from zero and infinity   and  as $z\to\infty$,
\begin{equation*}
\begin{aligned}
&\frac{\hatvarphiall[1](z)}{\widehat{{\cal M}}_{\nu,\alpha,\sigma^2_2}(z)}\\
&=\frac{\sigma_{1}^{2}\beta^{d}\Gamma(\nu)L^{\boldsymbol{\varsigma}}}{\Gamma(\nu+d/2)\sigma_{0}^{2}\alpha^{-2\nu}\pi^{-\frac{d}{2}}}\Big[c_{3}^{\boldsymbol{\varsigma}}(\beta z)^{-2\lambda}\big\{1+\mathcal{O}(z^{-2})\big\}
+ c_{4}^{\boldsymbol{\varsigma}}(z\beta)^{-(\mu+\lambda))}\\
&\quad\times\big\{\cos(\beta z-c_{5}^{\boldsymbol{\varsigma}})+\mathcal{O}(z^{-1})\big\}\Big](\alpha^{-2}+z^{2})^{\nu+\frac{d}{2}}\\
&=\frac{\sigma_{1}^{2}\beta^{d}\Gamma(\nu)L^{\boldsymbol{\varsigma}}}{\Gamma(\nu+d/2)\sigma_{0}^{2}\alpha^{-2\nu}\pi^{-\frac{d}{2}}}\Big[c_{3}^{\boldsymbol{\varsigma}}(\beta z)^{-2\lambda}\big\{1+\mathcal{O}(z^{-2})\big\}
+ c_{4}^{\boldsymbol{\varsigma}}(z\beta)^{-(\mu+\lambda))}\\
&\quad\times\big\{\cos(\beta z-c_{5}^{\boldsymbol{\varsigma}})+\mathcal{O}(z^{-1})\big\}\Big]z^{2\nu+d} \Big[1 + (\nu+d/2)(\alpha z)^{-2}+\mathcal{O}(z^{-2})\Big]\\
&=\frac{\sigma_{1}^{2}\beta^{d}\Gamma(\nu)L^{\boldsymbol{\varsigma}}}{\Gamma(\nu+d/2)\sigma_{0}^{2}\alpha^{-2\nu}\pi^{-\frac{d}{2}}}\Big[c_{3}^{\boldsymbol{\varsigma}}\beta ^{-2\lambda}z^{2\nu-2\lambda+d}\big\{1+\mathcal{O}(z^{-2})\big\}
+ c_{4}^{\boldsymbol{\varsigma}}\beta^{-(\mu+\lambda)}\\
&\quad\times z^{2\nu-(\mu+\lambda)+d}\big\{\cos(\beta z-c_{5}^{\boldsymbol{\varsigma}})+\mathcal{O}(z^{-1})\big\}\Big]\Big[1 + (\nu+d/2)(\alpha z)^{-2}+\mathcal{O}(z^{-2})\Big]
\end{aligned}
\end{equation*}
then, if $2\nu +d=2\lambda$, that is $\kappa+1/2=\nu$, $\mu>\lambda$ and considering Remark 1 then:
\begin{equation}\label{lim}
\underset{z\to\infty}{\lim}\frac{\hatvarphiall[1](z)}{\widehat{{\cal M}}_{\nu,\alpha,\sigma^2_2}(z)}=\frac{\sigma_1^{2}\beta_1^{-(1+2\kappa)}}{\sigma_{2}^{2}\alpha^{-2\nu}}\left(\mu \frac{\Gamma(2\kappa+\mu+1)}{\Gamma(\mu+1)}\right).
\end{equation}
%
Using Theorem~1  of \cite{Stein:1993}, we obtain~\eqref{kauf3_11}.
If $\sigma_1^{2}\beta_1^{-(1+2\kappa)}\left(\mu \frac{\Gamma(2\kappa+\mu+1)}{\Gamma(\mu+1)}\right)=\sigma_{2}^{2}\alpha^{-2\nu}$ and using Theorem~2 of \cite{Stein:1993},  we obtain~\eqref{kauf3_3}.
\end{proof}


The implication of Point~1 is that under $P( \varphiall[0])$,
prediction with $ \varphiall[1][0]$
with an arbitrary $\beta_1>0$
gives asymptotic prediction efficiency, if the correct value of
$\kappa$ and $\mu$ are used and $\mu>\lambda$.  By virtue of
Point~2, under $P({\cal M}_{\nu,\alpha,\sigma^2_2})$, prediction with
$\varphiall[1][0]$, with an arbitrary $\beta_1>0$, gives
asymptotic prediction efficiency, if $\nu=\kappa+{1}/{2}$,
$\mu>\lambda$.  For instance, if $\sigma_2^2 \e^{-r/\alpha}$ is
the true covariance, asymptotic prediction efficiency can be achieved
with $\sigma_0^2(1-r/\beta_1)^{\mu}_{+}$, using an arbitrary
$\beta_1$, and $\mu>1.5$ when $d=2$.
In view of Point~3, under $P(\varphiall[0])$, prediction with
$\varphiall[1]$, when
$\sigma_0^{2}\beta_0^{-(1+2\kappa)}=\sigma_1^{2}\beta_1^{-(1+2\kappa)}$
provides asymptotic prediction efficiency and asymptotically correct
estimates of error variance, if $\mu>\lambda$.  Finally, Point~4
implies that under $P( {\cal M}_{\nu,\alpha,\sigma_2^2})$, prediction
using $\varphiall[1]$, under the conditions $\mu
{\Gamma(2\kappa+\mu+1)}\big/{\Gamma(\mu+1)}
\sigma_{1}^{2}\beta_1^{-(1+2\kappa)}=\sigma_{2}^{2}\alpha^{-2\nu}$,
$\nu=\kappa+{1}/{2}$ and $\mu>\lambda$, provides asymptotic
prediction efficiency and asymptotically correct estimates of error
variance.

For instance, if $\sigma_2^2 \e^{-r/\alpha}$ is the true covariance and
$d=2$, asymptotic prediction efficiency and asymptotically correct
estimates of variance error can be achieved with
$\sigma_1^2(1-r/\beta_1)^{\mu}_{+}$ setting $\beta_1= \mu \alpha
{\sigma_1^2}{\sigma_2^{-2}}$, and $\mu>1.5$.  Setting
$\sigma^2_2=\sigma_1^2=1$, $\mu=3$, $\alpha=x/3$ ($x$ in this case is
the so-called practical range, i.e., the correlation is lower than
$0.05$ when $r>x$), the {\em equivalent} compact support is
$\beta_1=x$.
Note that in this special case, the practical range of the exponential
model and the compact support of the Askey function coincide.
Figure~\ref{cova} shows the
Mat\'ern  correlation model  with $\nu=0.5, 1, 1.5$ and  practical range equal to  $0.6$,
and two compatible GW correlation  models  when $d=2$ with $\kappa=\nu-0.5$, $\mu=\lambda+1+x$, with $x=0.5,2$  and the associated compact supports are obtained using the equivalence condition.
They are $0.601$,  $0.595$, $ 0.624$ for $\kappa=0, 0.5, 1$ respectively when $x=0.5$  and          $0.901$,  $0.821$,   $0.815 $
for $\kappa=0, 0.5, 1$, respectively,  when $x=2$.

\begin{figure}
\includegraphics[width=0.33\textwidth]{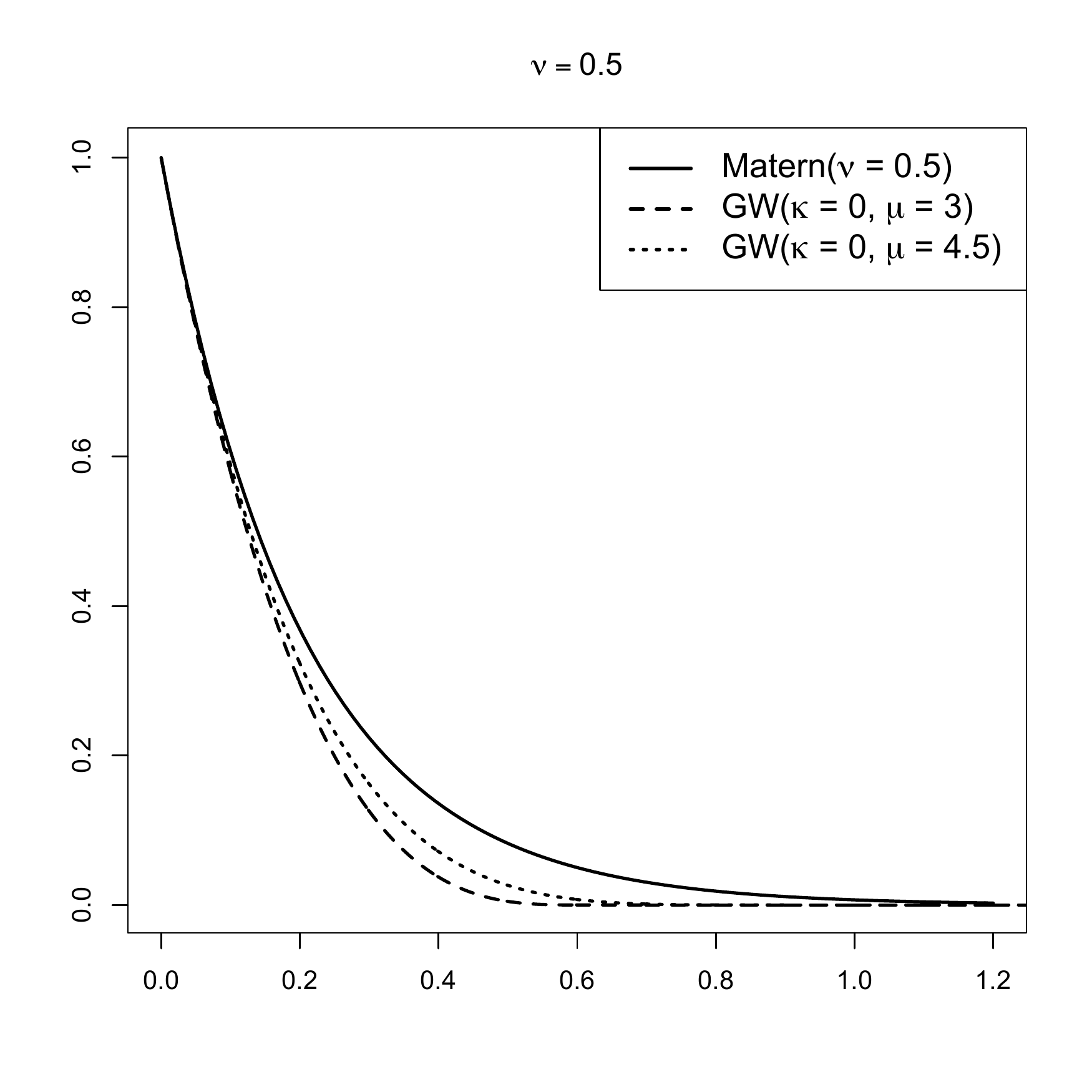}%
\hfill%
\includegraphics[width=0.33\textwidth]{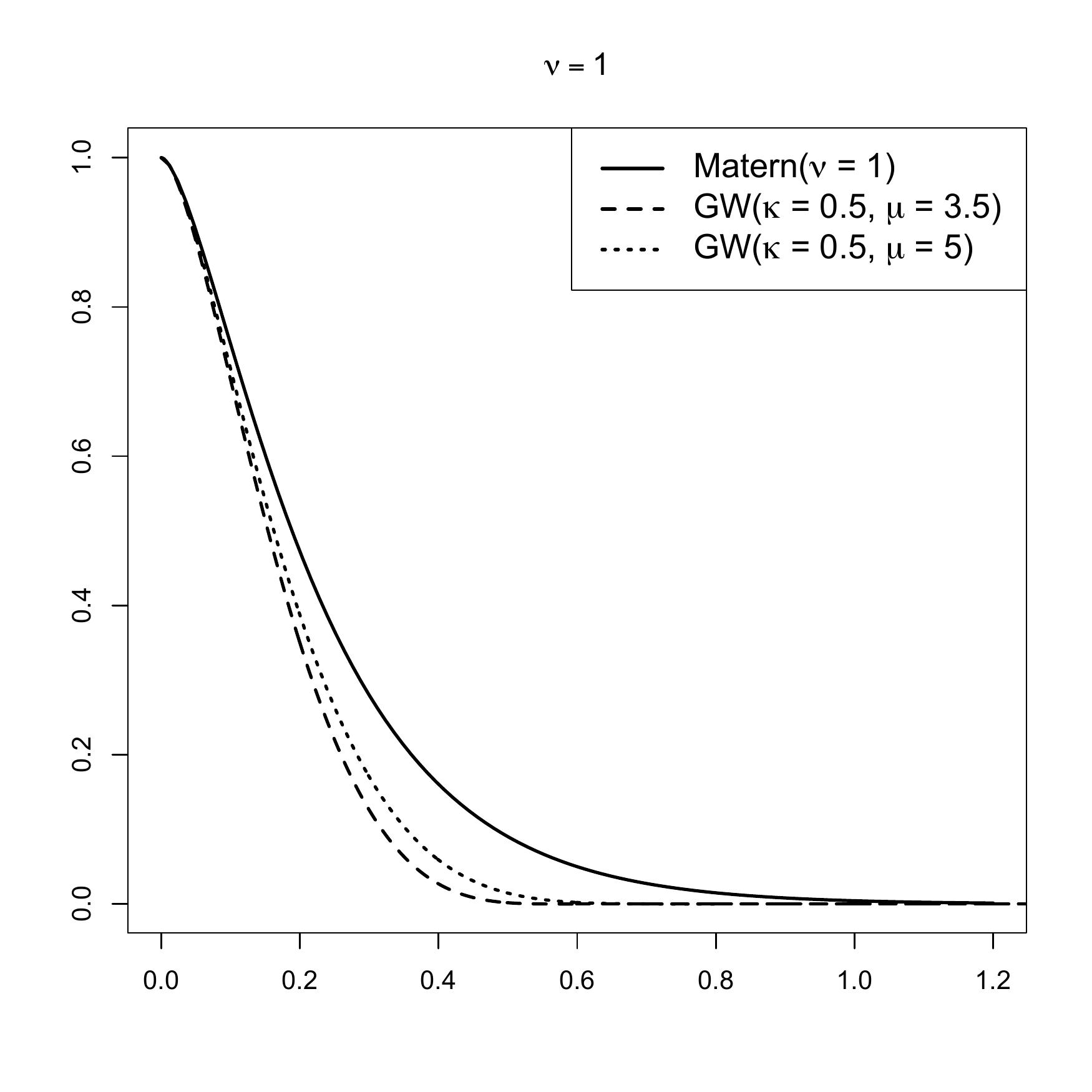}%
\hfill%
\includegraphics[width=0.33\textwidth]{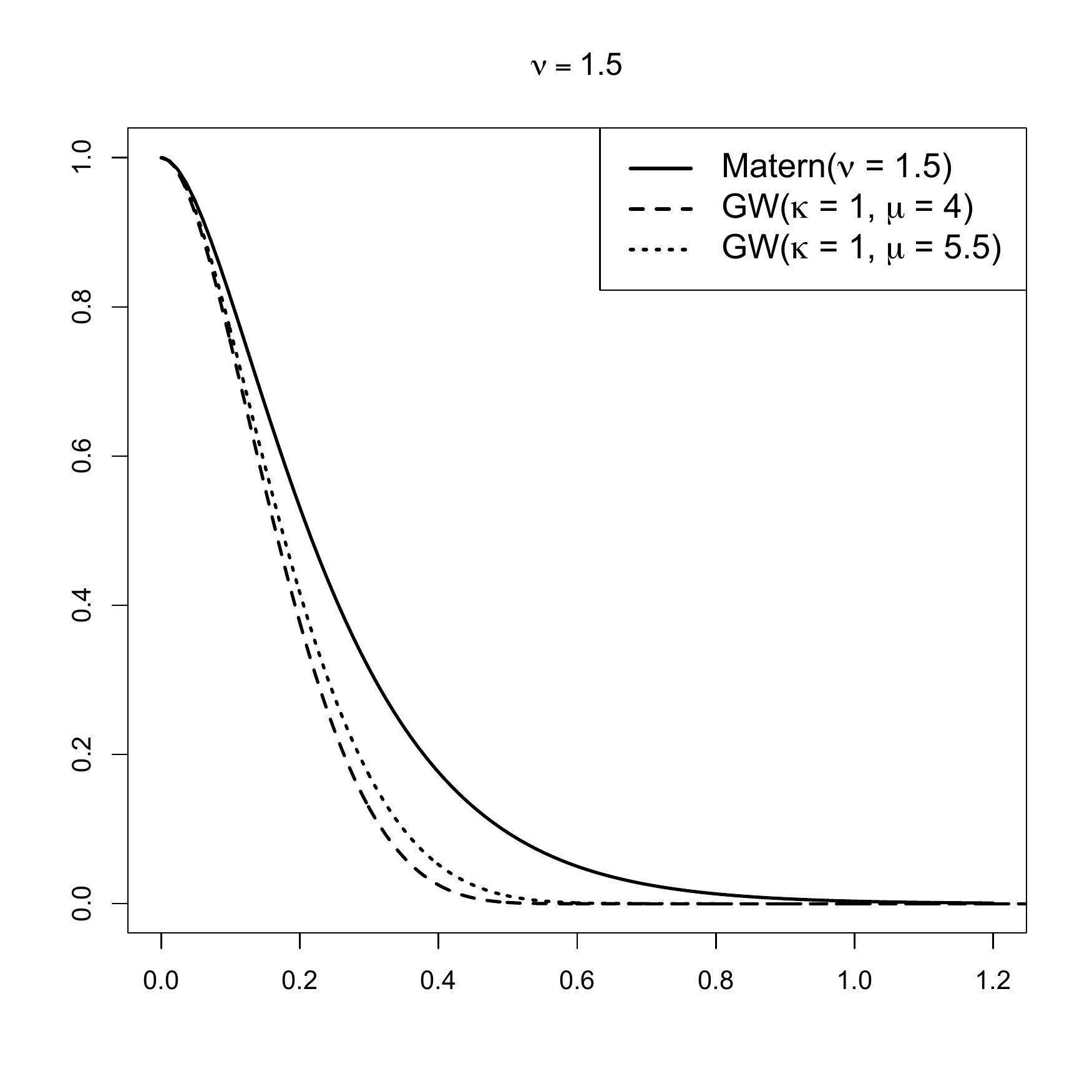}
\caption{Compatible correlation models for the case $d=2$:
The Matern model when $\nu=0.5, 1, 1.5$ (from left to right)  and the practical range  is $0.6$ and two compatibles GW models.
For the GW  models $\kappa=\nu-0.5$, $\mu=\lambda+1+x$, with $x=0.5,2$
and the compact support is fixed using the equivalence condition.}
 \label{cova}
\end{figure}

In practice, covariance parameters are unknown, so it is common to
estimate them and then plug into~\eqref{blup} and~\eqref{msetrue}.
Nevertheless, the asymptotic properties of this procedure are quite
difficult to obtain \citep{pu:ya:2001}. Instead, most theoretical
results have been given under a framework in which plug-in parameters
are fixed, rather than being estimated from observations.

As in Theorem~4 of \cite{Shaby:Kaufmann:2013},  our Points~3 and 4 may be extended to
include estimation of the variance parameter. Specifically let  $\widehat{\sigma}_{n}^{2}=\bZ_n^{\prime}R_{n}(\mu,\kappa,\beta_1)^{-1}\bZ_n/n$. Then
as $n\to \infty$,
\begin{align}
\label{kauf3_22} &\frac{\var_{\mu,\kappa,\beta_1,\widehat{\sigma}_n^2}\left[\widehat{Z}_{n}(\mu,\kappa,\beta_1)-Z(\boldsymbol{s}_0)\right]}{\var_{\mu,\kappa,\beta_0,\sigma^2_0}\left[\widehat{Z}_{n}(\mu,\kappa,\beta_1)-Z(\boldsymbol{s}_0)\right]}{\to}1, \\
\label{kauf3_33}&
 \frac{\var_{\mu,\kappa,\beta_1,\widehat{\sigma}_n^2}\left[\widehat{Z}_{n}(\mu,\kappa,\beta_1)-Z(\boldsymbol{s}_0)\right]}{\var_{\nu,\alpha,\sigma_2^2}\left[\widehat{Z}_{n}(\mu,\kappa,\beta_1)-Z(\boldsymbol{s}_0)\right]}{\to}1
. \end{align}
The proof follows the lines of  \cite{Shaby:Kaufmann:2013}, and we omit it.
As outlined in \cite{Shaby:Kaufmann:2013}, we also  conjecture that~\eqref{kauf3_22} and~\eqref{kauf3_33} hold if $\beta_1$ is replaced by its  maximum likelihood estimator.

\section{Simulations and illustrations}\label{6}
The main goals of this section are twofold: on the one hand, we compare
the finite sample behavior of the ML estimation of the microergodic
parameter of the GW model with the asymptotic distributions given in
Theorems~\ref{theo10} and~\ref{theo11}.
On the other hand, we  compare the finite sample behavior of MSE
prediction of a zero mean Gaussian field with Mat{\'e}rn covariance
model, using both a Mat{\'e}rn and a compatible GW covariance model, using  CT applied to a Mat{\'e}rn model as benchmark.

\smallskip

Regarding the first goal, we simulate, using Cholesky decomposition,
and then we estimate with ML, $1000$ realizations from a zero mean
Gaussian field with GW model.  Sampling locations are constructed as
in \cite{Kaufman:Schervish:Nychka:2008}, using a perturbed regular
grid. 
A perturbed grid helps to get more stable estimates because different sets of small distances are available to estimate the parameters.
Specifically, we have
considered a regular grid with increments $0.03$ over $[0,1]^d$,
$d=2$. Then the grid points have been perturbed, adding a uniform
random value on $[-0.01, 0.01]$ to each coordinate.
Figure~\ref{grid} shows the perturbed grid considered, from which we randomly choose  $n= 50, 100, 250, 500, 1000$ locations
without replacement.

\begin{figure}
\centering
\includegraphics[width=6cm]{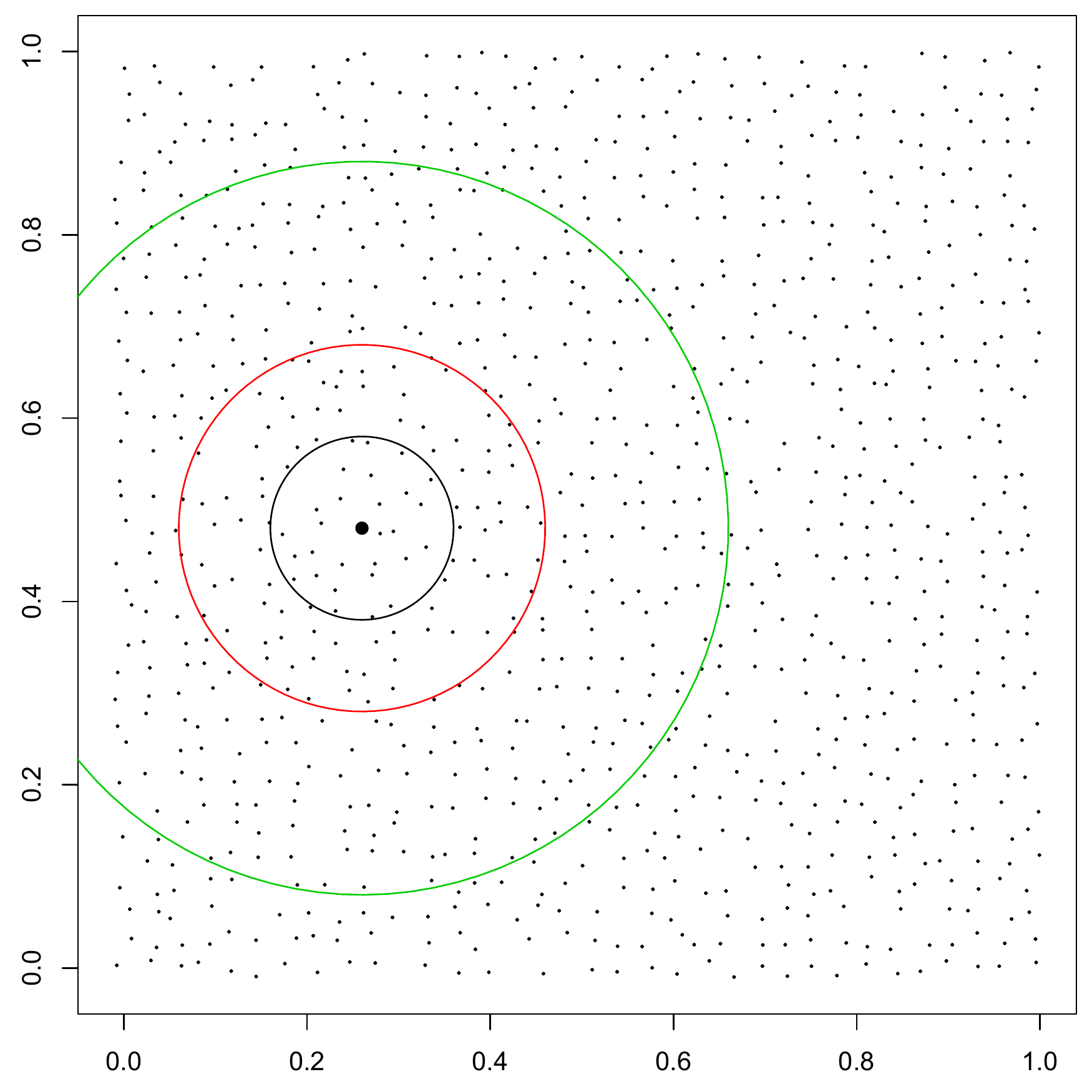}
\caption{Perturbated grid consisting of $n=1156$ considered in the simulation study. The black dot has coordinates $(0.26,0.48)$. In the circles (from smaller to larger) the location sites involved in  prediction with GW with compact support equal to $0.1, 0.2, 0.4$. } \label{grid}
\end{figure}

For the GW covariance model $\varphiall[0]$,
we use different values of the compact support and smoothness
parameters, that is $\beta_0=0.2, 0.4, 0.6$, $\kappa=0, 0.5, 1$, and fix $\sigma_0^2=1$ and, in view of Theorem~\ref{theo11},
$\mu=\lambda(2,\kappa)+3$.  For each simulation, we consider
$\kappa$ and $\mu$ as known and fixed, and we estimate with ML
the variance and compact support parameters, obtaining
$\hat{\sigma}^2_{i}$ and $\hat{\beta}_i$, $i=1,\ldots,1000$.
To estimate, we first  maximize the profile log-likelihood~\eqref{eq:prof} to
get $\hat{\beta}_i$. Then, we obtain $\hat{\sigma}_i^2(\hat{\beta}_i)=
\bz_i^{\prime}R(\hat{\beta}_i)^{-1}\bz_i/n$, where $\bz_i$ is the data
vector of simulation $i$.

Optimization was carried out using the R \citep{R:2005} function \emph{optimize}
 where,  following \cite{Shaby:Kaufmann:2013}, the
compact support parameter was restricted to the interval $
[\varepsilon, 15\beta]$ and $\varepsilon$ is slightly larger than
machine precision, about $10^{-15}$ here.

\long\def\tabletwo{

\begin{table}
  \caption{Sample quantiles, mean  and variance of $\sqrt{n/2}(\widehat{\sigma}_i^2(x)\beta_0^{1+2\kappa}/ (\sigma_0^2x^{1+2\kappa}) - 1)$, $i=1,\dots,1000$,
    for $x=\widehat{\beta}_i ,\beta_0, 0.5\beta_0, 2\beta_0$ for different values of $\kappa$, when $\beta_0=0.2$ and $n=250, 500, 1000$, compared with the associated theoretical
    values of the standard Gaussian distribution.}\label{tab21}
\medskip\centering
\scalebox{1}{
\begin{tabular}{|c|  c| c| c| c| c| c| c| c| c| }
	\hline
$\kappa$  &$x$& 	$n$ &  5$\%$ & 25$\%$ & 50$\%$ & 75$\%$ & 95$\%$ & Mean&Var \\
	\hline \multirow{12}{*}{0}
	&&250 & -2.026 &-0.856&  0.012&  1.143&  2.729&  0.177& 2.086 \\
&$\widehat{\beta}$	&500  &  -1.853 &-0.773 & 0.008  &0.949 & 2.198 & 0.104 &1.555 \\
	&&1000 &  -1.730 &-0.759& -0.010 & 0.802 & 1.880& 0.046 &1.255 \\
	\cline{2-10}
	&&250 &   -1.548& -0.670& -0.039&  0.675&  1.833   &0.025&1.058 \\
	&  $\beta_0$& 500  &-1.632& -0.665 & 0.001 & 0.661 & 1.754 &0.027& 1.047  \\
	&&1000&-1.629&-0.690 & 0.020 & 0.698 & 1.627 &0.011&1.009  \\
	\cline{2-10}
	&&250&6.432 & 8.325 & 9.611& 11.090& 13.243& 9.727&4.401  \\
	& $0.5\beta_0$& 500 &    8.893& 10.509 &11.958& 13.347& 15.469 &12.023& 4.131  \\
	&&1000&10.206 &12.130 &13.338& 14.626& 16.792 &13.427& 3.863 \\
	\cline{2-10}
	&&250 &-3.687& -3.016 &-2.491& -1.895& -0.982& -2.431& 0.693\\
&$2\beta_0$	&500& -3.694& -2.924& -2.341 &-1.773& -0.816& -2.315& 0.799  \\
	&&1000&-3.534 &-2.707& -2.096&-1.484& -0.607 & -2.087&0.842 \\
	\hline  \hline\multirow{12}{*}{0.5}
	&&250 & -2.260& -0.961 & 0.081 & 1.231&  3.252& 0.204& 2.671 \\
&$\widehat{\beta}$& 500 &-1.982 &-0.867 & 0.033  &1.062  &2.423  &0.107 &1.893  \\
	&&1000&  -1.802& -0.791 &-0.030 & 0.837  &2.006 & 0.037&1.377  \\
		\cline{2-10}
	&&250 & -1.548 &-0.670& -0.039 & 0.675 & 1.833 &0.025 &1.058  \\
    &$\beta_0$&500&-1.632 &-0.665 & 0.001 & 0.661 & 1.754 &0.027 &1.047  \\
	&&1000& -1.629 &-0.690 & 0.020 & 0.698 & 1.627 & 0.011&  1.009  \\
	\cline{2-10}
	&&250&21.663 &25.577 &28.182 &31.321 &35.457 &28.457 &18.068 \\
	   &$0.5\beta_0$& 500  &27.828& 31.502 &34.256& 37.217& 41.396 &34.432& 16.638  \\
	&&1000&31.135 &34.167 &36.694 &39.376 &43.438 &36.935& 14.650\\
	\cline{2-10}
	&&250& -5.149 &-4.594& -4.153& -3.671& -2.897& -4.105 &0.490\\
&$2\beta_0$&	500  &-5.251 &-4.546& -4.007 &-3.473& -2.605 & -3.979& 0.657  \\
	&&1000&-4.975 &-4.196 &-3.619 &-3.021& -2.207 &-3.603& 0.739 \\
	\hline  \hline\multirow{12}{*}{1}
	&&250 &-2.499 &-1.102&  0.071&  1.397&  3.669&  0.251& 3.575 \\
	&$\widehat{\beta}$&500  & -2.137 &-0.942 &0.006 &1.138  &2.738 &0.116& 2.309 \\
	&&1000 & -1.881 &-0.840& -0.013 & 0.851 & 2.126& 0.026 &1.511\\
	\cline{2-10}
	&&250 & -1.548 &-0.670 &-0.039 & 0.675 & 1.833 & 0.025& 1.058 \\
	&$\beta_0$&500  &-1.632 &-0.665  &0.001  &0.661 & 1.754 &0.027 &1.047 \\
	&&1000& -1.629 &-0.690  &0.020 & 0.698 & 1.627 & 0.011& 1.009  \\
	\cline{2-10}
	&&250 &52.484 &60.587 &65.940 &72.080 &80.575 &66.499& 73.714\\
&$0.5\beta_0$&	500 &67.395 &74.642& 80.127& 85.697& 94.740& 80.463&  68.423\\
	&&1000& 74.224 &80.409& 85.327& 90.627 &98.917 & 85.827& 59.410 \\
	\cline{2-10}
	&&250 &-6.576& -6.110 &-5.745 &-5.352 &-4.732&  -5.705& 0.325 \\
 &$2\beta_0$&	500 & -6.840 &-6.247 &-5.787 &-5.262 &-4.544 & -5.740& 0.517 \\
	&&1000& -6.563 &-5.848& -5.330& -4.771& -4.021 & -5.304 & 0.630  \\
	\hline \hline
		\multicolumn{3}{|c|}{$N (0,1)$} &-1.645 & -0.674 &  0 &  0.674 & 1.645&0&1\\
	\hline
\end{tabular}
}
\end{table}

}


\begin{table}
  \caption{Sample quantiles, mean and variance of  $\sqrt{n/2}(\widehat{\sigma}_i^2(x)\beta_0^{1+2\kappa}/ (\sigma_0^2x^{1+2\kappa}) - 1)$, $i=1,\dots,1000$,
    for $x=\widehat{\beta}_i ,\beta_0, 0.5\beta_0, 2\beta_0$ for different values of $\kappa$, when $\beta_0=0.4$ and $n=250, 500, 1000$, compared with the associated theoretical
    values of the standard Gaussian distribution.}\label{tab22}
\medskip\centering
\scalebox{1}{
\begin{tabular}{|c|  c| c| c| c| c| c| c| c| c| }
	\hline
$\kappa$  &$x$& 	$n$ &  5$\%$ & 25$\%$ & 50$\%$ & 75$\%$ & 95$\%$ & Mean&Var \\
	\hline \multirow{12}{*}{0}
	&&250 & -1.699 &-0.721& -0.020  &0.798  &2.084 &0.072 &1.375 \\
&$\widehat{\beta}$	&500  & -1.680 &-0.677  &0.027  &0.758  &1.966  &0.071 &1.212 \\
	&&1000 & -1.614& -0.666 & 0.062  &0.767 & 1.788 &0.057 &1.104 \\
	\cline{2-10}
	&&250&-1.548& -0.670& -0.039 & 0.675 & 1.833 & 0.025& 1.058  \\
	&  $\beta_0$& 500  & -1.632& -0.665& 0.001 & 0.661 & 1.754 &0.027& 1.047  \\
	&&1000&    -1.629& -0.690 & 0.020&  0.698 & 1.627 & 0.011& 1.009  \\
	\cline{2-10}
	&&250 &3.224 &4.953& 6.163& 7.471& 9.370 &6.234 &3.493 \\
	& $0.5\beta_0$& 500 &   3.399& 4.762& 5.948&7.018& 8.879&  5.979 & 2.840  \\
	&&1000&   2.792&4.063& 5.059& 5.984&7.516 &5.088 & 2.088  \\
	\cline{2-10}
	&&250 &-2.443& -1.698 &-1.128&-0.490 & 0.610&-1.065& 0.898 \\
&$2\beta_0$	&500& -2.485& -1.576&-0.941& -0.313 &0.718 & -0.904& 0.947  \\
	&&1000& -2.324 &-1.438 &-0.759& -0.107 & 0.819&  -0.757 &0.949  \\
	\hline  \hline\multirow{12}{*}{0.5}
	&&250 & -1.761& -0.786&  0.019&  0.807  &2.271&  0.072& 1.506  \\
&$\widehat{\beta}$& 500 & -1.774 &-0.714  &0.027  &0.822  &1.978 &0.063 &1.309 \\
	&         &1000&  -1.609 &-0.700 & 0.047  &0.761  &1.840 &0.051 &1.152 \\
		\cline{2-10}
		&&250 & -1.548 &-0.670 &-0.039  &0.675  &1.833 & 0.025 &1.058  \\
    &$\beta_0$&500& -1.632 &-0.665 & 0.001 & 0.661 & 1.754& 0.027& 1.047  \\
	&&1000&-1.629 &-0.690&  0.020&  0.698 & 1.627& 0.011& 1.009  \\
	\cline{2-10}
		&&250& 11.462 &14.603& 16.995 &19.573& 23.414 &17.155& 12.818\\
	   &$0.5\beta_0$& 500  &11.133 &13.624 &15.459& 17.592& 21.090&   15.697& 9.060\\
	&&1000&  9.192 &11.051 &12.578 &14.187 &16.904  &12.733 &5.560\\
	\cline{2-10}
		&&250& -3.166 &-2.469& -1.914& -1.315 &-0.260&  -1.860& 0.784\\
&$2\beta_0$&	500  & -3.136 &-2.258 &-1.628 &-1.037 &-0.029 &-1.604  &0.883  \\
	&&1000&-2.851 &-1.999 &-1.353 &-0.707  &0.207 &-1.342 &0.907\\
	\hline  \hline\multirow{12}{*}{1}
	&&250 &-1.825& -0.868&  0.042&  0.836&  2.389&  0.078& 1.661  \\
	&$\widehat{\beta}$&500  & -1.869& -0.770 & 0.027&  0.820&  2.092& 0.059& 1.412 \\
	&&1000 & -1.679& -0.719&  0.058&  0.762&  1.836&   0.045 & 1.199 \\
	\cline{2-10}
	&&250 & -1.548 &-0.670& -0.039 & 0.675 & 1.833 &0.025 &1.058 \\
	&$\beta_0$&500  &-1.632 &-0.665 & 0.001 & 0.661 & 1.754 &0.027 &1.047\\
	&&1000& -1.629 &-0.690&  0.020 & 0.698 & 1.627 & 0.011& 1.009  \\
	\cline{2-10}
	&&250 &28.654 &34.704 &39.574 &44.651 &52.477 & 39.856 & 51.483 \\
&$0.5\beta_0$&	500 &27.166 &31.848 &35.553& 39.808& 46.519&  35.992& 34.995  \\
	&&1000& 22.055 &25.398 &28.218 &31.256 &36.451 &28.565 &19.929  \\
	\cline{2-10}
	&&250 & -3.949 &-3.312 &-2.806 &-2.262 &-1.288  & -2.750& 0.666 \\
 &$2\beta_0$&	500 &-3.876 &-3.050& -2.445& -1.862& -0.925 &-2.427& 0.809  \\
	&&1000&-3.524& -2.675& -2.065& -1.419& -0.532& -2.047& 0.856  \\
	\hline \hline
		\multicolumn{3}{|c|}{$N (0,1)$} &-1.645 & -0.674 &  0 &  0.674 & 1.645&0&1\\
	\hline
\end{tabular}

}
\end{table}

\long\def\tablefour{

\begin{table}
\caption{Sample quantiles, mean  and variance of  $\sqrt{n/2}(\widehat{\sigma}_i^2(x)\beta_0^{1+2\kappa}/ (\sigma_0^2x^{1+2\kappa}) - 1)$, $i=1,\dots,1000$,
 for $x=\widehat{\beta}_i ,\beta_0, 0.5\beta_0,2\beta_0$ for different values of $\kappa$, when $\beta_0=0.6$ and $n=250, 500, 1000$, compared with the associated theoretical
 values of the standard Gaussian distribution.}\label{tab23}
\centering
\medskip
\scalebox{1}{
\begin{tabular}{|c|  c| c| c| c| c| c| c| c| c| }
	\hline
$\kappa$  &$x$& 	$n$ &  5$\%$ & 25$\%$ & 50$\%$ & 75$\%$ & 95$\%$ & Mean&Var \\
	\hline \multirow{12}{*}{0}
	&&250 &  -1.572 &-0.700 &-0.048  &0.748  &1.952 &0.065 &1.248 \\
&$\widehat{\beta}$	&500  &  -1.657 &-0.674 & 0.027 & 0.747 & 1.924& 0.072& 1.160 \\
	&&1000 &-1.599 &-0.664 & 0.052  &0.749  &1.777  &0.058 &1.072 \\
	\cline{2-10}
	&&250 &  -1.548& -0.670& -0.039 & 0.675 & 1.833 &0.025 &1.058 \\
	&  $\beta_0$& 500  &   -1.632 &-0.665 & 0.001 & 0.661 & 1.754 & 0.027& 1.047  \\
	&&1000&   -1.629 &-0.690 & 0.020& 0.698 & 1.627 & 0.011 &1.009\\
	\cline{2-10}
	&&250 &   1.394 &2.829 &3.871& 5.008 &6.761 &3.946  &2.625  \\
	& $0.5\beta_0$& 500 &  1.080& 2.369 &3.293& 4.218& 5.924&3.357 & 2.035 \\
	&&1000& 0.661 &1.801& 2.644 &3.430&4.746 &2.646 &1.532  \\
	\cline{2-10}
	&&250 & -2.113 &-1.285& -0.697& -0.037 & 1.150 &-0.623 & 0.964  \\
&$2\beta_0$	&500&-2.120 &-1.182 &-0.543 & 0.112 &1.172 & -0.504 &  0.994\\
	&&1000& -1.998& -1.099& -0.426 & 0.252 & 1.173&  -0.413 &0.977 \\
	\hline  \hline\multirow{12}{*}{0.5}
	&&250 & -1.594& -0.714& -0.037&  0.726  &2.012 & 0.061& 1.318 \\
&$\widehat{\beta}$& 500 & -1.656 &-0.690  &0.012 & 0.784 & 1.914  &0.065 &1.210 \\
	&&1000        & -1.606 & -0.661 & 0.040 & 0.765 & 1.777 & 0.054 & 1.100 \\
		\cline{2-10}
	&&250 &-1.548& -0.670& -0.039 & 0.675 & 1.833 &0.025 &1.058  \\
    &$\beta_0$&500& -1.632 &-0.665 & 0.001&  0.661&  1.754& 0.027& 1.047 \\
	&&1000& -1.629& -0.690&  0.020&  0.698 & 1.627& 0.011& 1.009 \\
	\cline{2-10}
	&&250& 5.907&  8.183 &10.117& 12.150& 15.293 &10.316  &8.162\\
	   &$0.5\beta_0$& 500  & 4.963  &6.887  &8.341  &9.853 &12.544  &8.455 &5.145\\
	&&1000& 3.738 &5.232 &6.415 &7.554 &9.443  &6.475 &3.173 \\
	\cline{2-10}
	&&250& -2.522 &-1.767& -1.167& -0.555 & 0.598& -1.112 & 0.89\\
&$2\beta_0$&	500  & -2.486 &-1.589 &-0.950 &-0.319  &0.708 &-0.915& 0.957 \\
	&&1000& -2.301 &-1.402 &-0.747 &-0.103 & 0.831 &-0.745 &0.954  \\
	\hline  \hline\multirow{12}{*}{1}
	&&250 & -1.642 &-0.742& -0.057&  0.717&  2.022&  0.059& 1.389  \\
	&$\widehat{\beta}$&500  & -1.719 &-0.691 &-0.001 & 0.803  &1.955 &0.059 &1.262 \\
	&&1000 &-1.647 &-0.689  &0.049  &0.756  &1.807  &0.049 &1.127 \\
	\cline{2-10}
	&&250 & -1.548 &-0.670 &-0.039  &0.675  &1.833  &0.025& 1.058  \\
	&$\beta_0$&500  &-1.632 &-0.665 & 0.001 & 0.661 & 1.754 &0.027& 1.047 \\
	&&1000& -1.629 &-0.690&  0.020  &0.698 & 1.627& 0.011& 1.009 \\
	\cline{2-10}
	&&250 &15.249 &19.717 &23.106 &27.216 &33.425 &23.612 &31.145  \\
&$0.5\beta_0$&	500 & 12.769& 16.135 &18.785 &21.548& 26.854 &19.080 & 17.975  \\
	&&1000&  9.750& 12.241 &14.111 &16.232 &19.773 &14.395& 10.058  \\
	\cline{2-10}
	&&250 &-3.004 &-2.306& -1.735& -1.145 &-0.077& -1.686& 0.808   \\
 &$2\beta_0$&	500 & -2.945& -2.083& -1.450& -0.829 & 0.148& -1.416& 0.911  \\
	&&1000& -2.692 &-1.819 &-1.170 &-0.531 & 0.419   &-1.157  &0.925 \\
	\hline \hline
		\multicolumn{3}{|c|}{$N (0,1)$} &-1.645 & -0.674 &  0 &  0.674 & 1.645&0&1\\
	\hline
\end{tabular}

}
\end{table}

}

Using the asymptotic distributions stated in Theorems~\ref{theo10}
and~\ref{theo11}, Table~\ref{tab22} compares
the sample quantiles of order $0.05, 0.25, 0.5, 0.75, 0.95$,
mean and variance of
$ \sqrt{n/2}\big(\widehat{\sigma}_i^2(x)\beta_0^{1+2\kappa}/(\sigma_0^2x^{1+2\kappa} )-1   \big)$
for
$x=\widehat{\beta}_i ,\beta_0, 0.5\beta_0, 2\beta_0$ with the
associated theoretical values of the standard Gaussian distribution,
for $\beta_0=0.4$, 
$\kappa=0, 0.5, 1$ and $n=250, 500, 1000$.

%

As expected, the best approximation is achieved overall when using the
true compact support, i.e., $x=\beta_0$, with little difference between the different values of $\beta$ and $\kappa$.
In the case of
$x=\widehat{\beta}_i$, the asymptotic distribution given in
Theorem~\ref{theo11} is a satisfactory approximation of the sample
distribution, visually improving when increasing $n$.
The value of $\kappa$ has less impact
compared to $\beta_0$. In general, smaller values lead to better
results.

When using compact supports that are too small or too large with respect to
the true compact support ($x= 0.5\beta_0, 2\beta_0$), the convergence
of the asymptotic distribution given in Theorem~\ref{theo10} is very
slow. In particular, when $x=0.5\beta_0$, the asymptotic approximation is not
satisfactory even for $n=1000$.  In other words,  confidence intervals for the
microergodic parameter, based on Theorem~\ref{theo10}, i.e., fixing an
arbitrary compact support, can be problematic when applied to finite
samples, even for large sample sizes.  We strongly recommend
jointly estimating variance and compact support and using the asymptotic
distribution give in Theorem~\ref{theo11} or, alternatively, choosing $\beta$ conservatively.

As a graphical example, Figure~\ref{cdf} compares the empirical CDF of
the ML estimates of the standardized microergodic parameter with the
CDF of the standard Gaussian distribution when $\sigma^2_0=1$,
$\kappa=0, 0.5, 1$, $\beta_0=0.6$ and $n=250, 500, 1000$.  Finally, our numerical results
are consistent with the results in \cite{Shaby:Kaufmann:2013}, in the
Mat{\'e}rn case.

\begin{figure}
\includegraphics[width=0.33\textwidth]{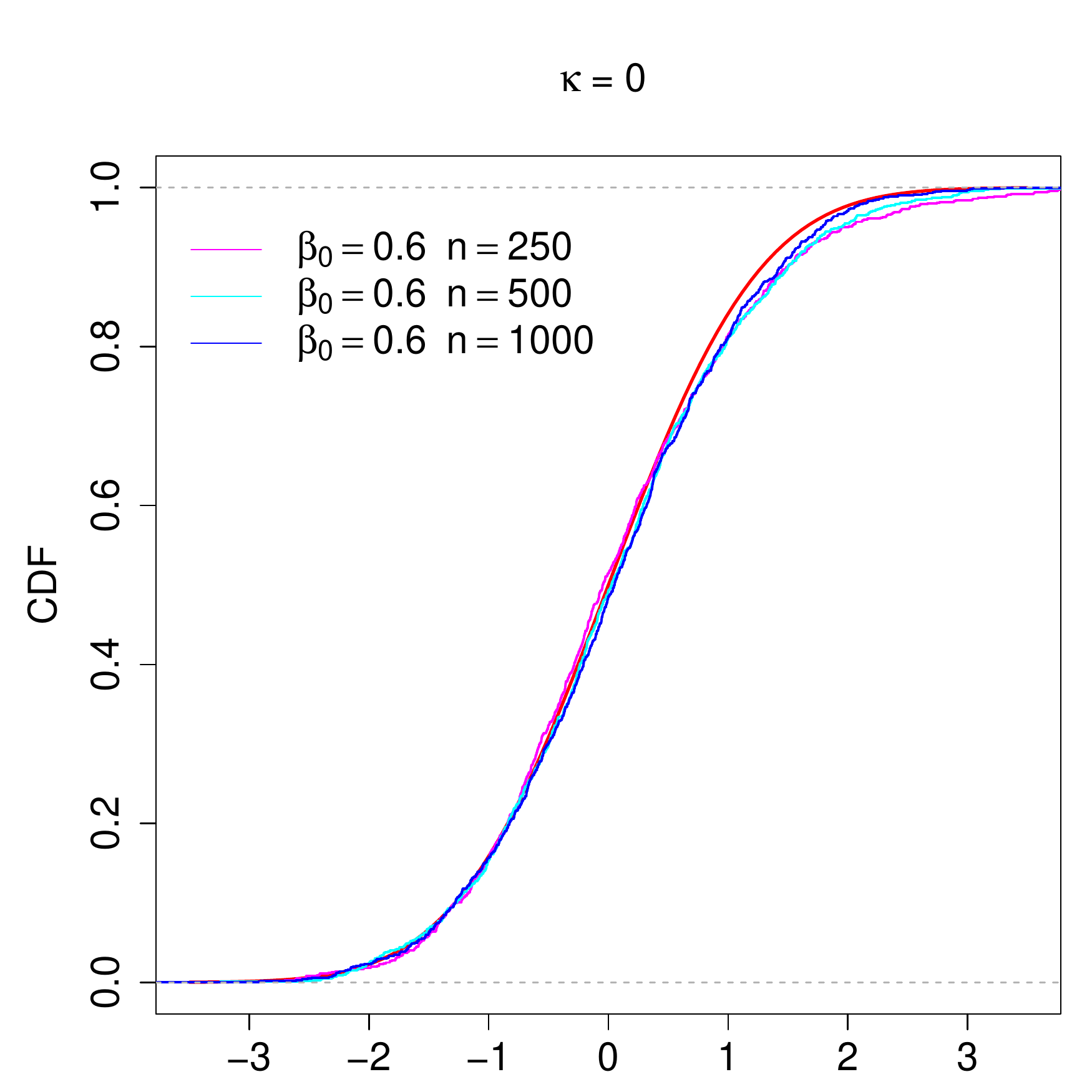}%
\hfill%
\includegraphics[width=0.33\textwidth]{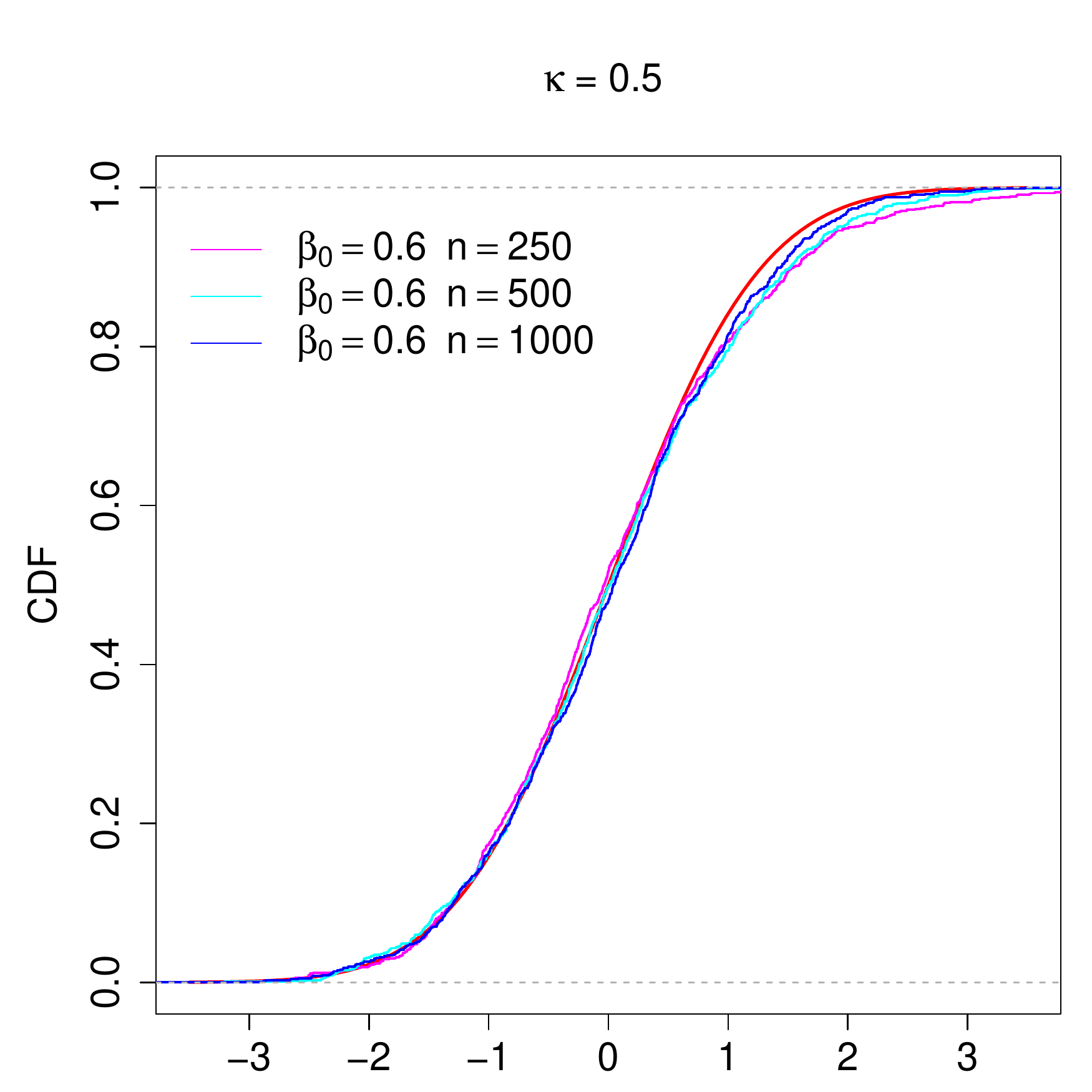}%
\hfill%
\includegraphics[width=0.33\textwidth]{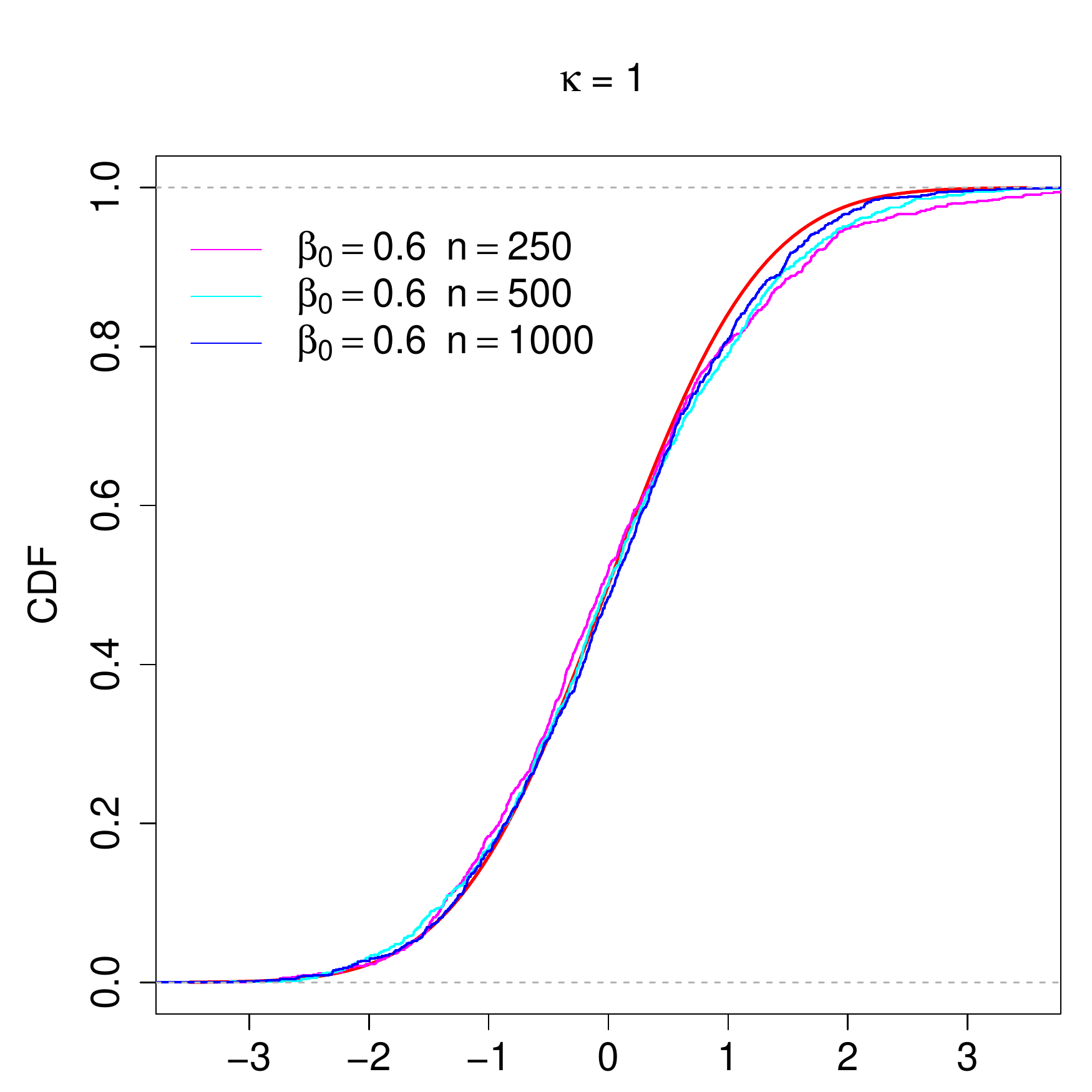}
\caption{Empirical CDF of the simulated ML estimation of the  standardized microergodic parameter
   vs CDF of a standard Gaussian distribution (red line) when $\sigma_0^2=1$, $\kappa=0, 0.5,1$ (from left to right), $\beta_0=0.6$
  and $n=250, 500, 1000$}\label{cdf}
\end{figure}

As for the second goal, using the results given in
Theorem~\ref{kauf_3} Points~2 and~4, we now specifically compare
asymptotic
prediction efficiency
 and asymptotically correct
estimation of prediction variance using ratios~\eqref{kauf3_11}
and~\eqref{kauf3_3} respectively.  As a benchmark, we also consider the
same ratios using a tapered Mat{\'e}rn model.  

More precisely, we consider a   Mat{\'e}rn model ${\cal M}_{\nu,\alpha,\sigma^2_2}$ setting
$\sigma_2^2=1$, $\nu=0.5, 1, 1.5$ and  $\alpha=y/c_{\nu}$ with $y=0.1,0.2,0.4$ if $\nu=0.5$,
$y=0.101,0.202,0.404$ if $\nu=1$ and $y=0.097,0.193,0.385$ if $\nu=1.5$.
Here  $c_{\nu}$ is a scalar depending on $\nu$ such that ${\cal M}_{\nu,1,1}(r)$
is lower than 0.05 when $r>c_{\nu}$ that is, $y$ is the practical range.

Let us define the ratios~\eqref{kauf3_11} and~\eqref{kauf3_3} as
$\U_1(\beta_1$) and $\U_2$,  respectively.  For each $\nu$ and $\alpha$,
we randomly select $n_j= 50, 100, 250, 500, 1000$, $j=1,\ldots,500$
location sites without replacement from the perturbed grid in
Figure~\ref{grid}.
For each $j$, we compute the ratio $\U_{1j}(\beta_1$) and the ratio
$\U_{2j}$, $j=1,\ldots,500$, using closed form expressions in
Equation~\eqref{mse_miss} and~\eqref{msetrue} when predicting the
location site $(0.26,0.48)^{\prime}$ (black dot in Figure~\ref{grid}).
Specifically for each $\U_{2j}$, following the conditions in
Theorem~\ref{kauf_3} Point~4, we set $\sigma_1^2=1$, $\kappa=\nu-1/2$,
 $\mu=\lambda+1.5$. The ``equivalent'' compact support is obtained as:
\begin{equation*}
   \beta_1^*=\left[\left(\mu \frac{\Gamma(2\kappa+\mu+1)}{\Gamma(\mu+1)} \right)\frac{\sigma_1^2  \alpha^{-2\nu}}{\sigma_2^2}\right]^{1/(1+2\kappa)}.
\end{equation*}
Under this specific setting the ``equivalent'' compact support
associated to the (varying with $\nu$) practical range is
approximately $\beta_1^*=0.1, 0.2, 0.4$, irrespectively of $\nu$.
Figure~\ref{grid} shows the location sites involved in the prediction using GW functions with $\beta_1^*=0.1, 0.2, 0.4$.

For each $\U_{1j}(\beta)$, following Theorem~\ref{kauf_3} Point~2, we
fix $\kappa=\nu-1/2$, $\mu=\lambda+1.5$ and $\beta=\beta_1^*$.  Then,
to investigate the effect of considering an arbitrary compact
support on the convergence of ratio~\eqref{kauf3_11}, we also
consider, $\U_{1j}(0.2\beta_1^*)$ and $\U_{1j}(5\beta_1^*)$.
For each combination of $\nu$, $\alpha$, Table~\ref{ratios} shows the
empirical means $\bar{\U}_{1}(x\beta_1^*)=\sum_{j=1}^{500}
\U_{1j}(x\beta_1^*)/500$ for $x=1, 0.5, 2$, and
$\bar{\U}_{2}=\sum_{j=1}^{500} \U_{2j}/500$ when increasing $n$.

As a benchmark, we also compute the empirical means replacing the GW
model with a tapered Mat{\'e}rn covariance model, that is, considering
the model ${\cal M}_{\nu,\alpha,\sigma^2_2} K_{x\beta_1^*}$, and we
denote these means by $\bar{\U}^T_{1}(x\beta_1^*)$, $x=1, 0.5, 2$ and
$\bar{\U}^T_{2}$.  Here, $K_{x\beta_1^*}$ is a known compactly supported
correlation function called taper function.  Following
\cite{Furrer:2006}, as taper function, we use
$K_{x\beta_1^*}=\varphi_{2,0,x\beta_1^*,1}$ if $\nu=0.5$,
$K_{x\beta_1^*}=\varphi_{3,1,x\beta_1^*,1}$ if $\nu=1$ and
$K_{x\beta_1^*}=\varphi_{4,2,x\beta_1^*,1}$ if $\nu=1.5$.  for $x=1,
0.5, 2$.

These specific choices of taper functions guarantee the
convergence of ratios~\eqref{kauf3_11} and~\eqref{kauf3_3}, using a
tapered Mat{\'e}rn model instead of the GW model (see Theorem~2 in
\citealp{Furrer:2006}).  In Table~\ref{ratios}, the percentages of
nonzero elements in the covariance matrices are also reported in all
scenarios and for each $n$ when using the compact support $\beta_1^*$.

Table~\ref{ratios} shows that $\bar{\U}_2$ clearly overall outperforms
$\bar{\U}_2^{T}$ in terms of speed of convergence in particular when
increasing $\beta_1^*$.  This implies that in terms of finite sample,
if the Matern model is the state of nature, prediction efficiency and
correct estimation of prediction variance are better achieved when
predicting with the (compatible) GW model with respect to the
so-called naive CT predictor \citep{Furrer:2006}, sharing the same
compact support.

Comparing $\bar{\U}_{1}(x\beta_1^*)$ with $\bar{\U}^T_{1}(x\beta_1^*)$
for $x=1, 0.5, 2$ note that when $x=1$, $\bar{\U}_{1}(\beta_1^*)$
overall slightly outperforms $\bar{\U}^T_{1}(\beta_1^*)$ and when
$x=0.5$, the convergence of both ratios seems to be very slow, in
particular for larger $\nu$. This suggests that taking an
arbitrary compact support too small with respect to the ``equivalent''
compact support $\beta_1^*$ can seriously affect the prediction
efficiency both for tapered Mat{\'e}rn and GW models. This kind of
problem disappears when $x=2$, as expected.
By the tapering effect, i.e., inducing a covariance with an apparent shorter range,   $\bar{\U}^T_{1}(2\beta_1^*)$ slightly outperforms
$\bar{\U}_{1}(2\beta_1^*)$.



\section{Concluding Remarks}\label{7}
Parameter estimation for
 interpolation of spatially
or spatio-temporally correlated random processes is used in many areas
and often requires particular models or careful implementation. In
recent years the dataset sizes have steadily increased such that
straightforward statistical tools are computationally too expensive to
use. The use of covariance functions with an (inherent or induced)
compact support, leading to sparse matrices, is a very accessible and
scalable approach.  In this paper we studied estimation and prediction
of Gaussian fields with covariance models belonging to the GW class,
under fixed domain asymptotics.

Specifically, we first characterize the equivalence of two Gaussian
measures with GW models, and then we establish strong
consistency and asymptotic Gaussianity of the ML estimator of the associated microergodic
parameter when considering both an arbitrary and an estimated compact
support.  Simulation results show that for a finite sample, the choice
of an arbitrary compact support can result in a very poor
approximation of the asymptotic distribution.  These results are
consistent with those in \cite{Shaby:Kaufmann:2013} in the Mat{\'e}rn
case.

In a second aspect, we give a sufficient condition for the equivalence of two
Gaussian measures with Mat{\'e}rn and GW model and we
study the effect on prediction when using these two covariance
models under fixed domain asymptotics.
A first consequence of our results is that GW model
is more than a valid competitor of the Mat{\'e}rn model. It allows, as
in the Mat{\'e}rn case, a continuous parameterization of smoothness of
the underlying Gaussian field and, under fixed domain asymptotics,
prediction and mean square error prediction obtained with a Mat{\'e}rn
model can be achieved using a GW model inducing an equivalent
Gaussian measure, using our condition.
For this reason, we advocate the GW class when working with (not necessarily) large or
huge spatial datasets since well established and implemented algorithms for sparse matrices can be
used when estimating the covariance parameters and/or predicting at
unknown locations (e.g., \citealp{Furrer:Sain:2010}).
Alternatively, for  covariances which are analytic away from the origin as  the Mat{\'e}rn model, 
 in some  circumstances a hierarchical factorization scheme as proposed for instance in  \cite{Ambikasaran:2016},
 is a possible solution 
in order to handle  sample sizes that cannot be handled by straightforward Cholesky factorization.

As the theoretical and numerical results illustrate, CT for prediction
is essentially an obsolete approach. When comparing both approaches with the same sensible compact support, the tapered CT is less efficient.
For estimation, one has to
distinguish between a so-called one-taper or two-taper approach, i.e.,
a proper likelihood or an estimating function approach,
\cite{Kaufman:Schervish:Nychka:2008}.
Fixing again the support, a GW model can approximate a Mat\'ern covariance function much better than a tapered one.
Thus, the GW is in an estimation setting superior to a one-taper CT.
In both approaches, one needs to be aware of the resulting
biases, which can be substantial. In the case of (kriging) predictions
based on plug-in estimates, the biases are largely canceled
\citep{Furr:Bach:Du:16}.
Finally, the two-taper approach is conceptually a different approach and, as it is  computationally very expensive, it would not be fair to compare it with the GW model.

Similarly to the Mat\'ern model with smoothness parameter different
to $p+1/2$, $p\in\mathbb{N}$, the GW does not have a closed form expression when its smoothness parameter is different to $p$,
and low level software implementations are needed for a
computationally efficient use.

\appendix


\bibliographystyle{imsart-nameyear.bst}
\bibliography{mybib}

\pagestyle{empty}

\begin{sidewaystable}
\caption{$\bar{\U}_1(x)$, $\bar{\U}^T_1(x)$, $x=0.5\beta_1^*, 2\beta_1^*,\beta_1^*$ and $\bar{\U}_2$, $\bar{\U}^T_2$, as defined in Section~\ref{6},
when considering a Mat{\'e}rn model  with  increasing practical range $y$,  smoothness parameter $\nu$  and $n$.
Here $\beta_1^*$ is the  compact support parameter  of the GW model computed using the  equivalence condition.
The column $\%$ reports the  mean of percentages of non-zero elements in the covariance matrices involved when considering $\beta_1^*$.
}
 \label{ratios}
\medskip\centering\scriptsize
\tabcolsep2pt\hspace*{-12mm}\begin{tabular}{|r|l|rrrrrrrr|rrrrrrrr|rrrrrrrr|r|}
  \hline
  && \multicolumn{8}{|c|}{$\nu=0.5$}&\multicolumn{8}{|c|}{$\nu=1$}&\multicolumn{8}{|c|}{$\nu=1.5$}&\raisebox{0pt}[11pt][7pt]{}\\
\hline
&$n$   &\rotatebox{90}{$\bar{\U}_1(0.5\beta_1^*)$}&\rotatebox{90}{$\bar{\U}_1(2\beta_1^*)$}&\rotatebox{90}{$\bar{\U}_1(\beta_1^*)$}&\rotatebox{90}{$\bar{\U}^T_1(0.5\beta_1^*)$~}&\rotatebox{90}{$\bar{\U}^T_1(2\beta_1^*)$}&\rotatebox{90}{$\bar{\U}_1^{T}(\beta_1^*)$}&\rotatebox{90}{$\bar{\U}_2$}&\rotatebox{90}{$\bar{\U}^{T}_{2}$}  &\rotatebox{90}{$\bar{\U}_1(0.5\beta_1^*)$}&\rotatebox{90}{$\bar{\U}_1(2\beta_1^*)$}&\rotatebox{90}{$\bar{\U}_1(\beta_1^*)$}&\rotatebox{90}{$\bar{\U}^T_1(0.5\beta_1^*)$}&\rotatebox{90}{$\bar{\U}^T_1(2\beta_1^*)$~}&\rotatebox{90}{$\bar{\U}_1^{T}(\beta_1^*)$} &\rotatebox{90}{$\bar{\U}_2$}&\rotatebox{90}{$\bar{\U}^{T}_{2}$} &\rotatebox{90}{$\bar{\U}_1(0.5\beta_1^*)$}&\rotatebox{90}{$\bar{\U}_1(2\beta_1^*)$}&\rotatebox{90}{$\bar{\U}_1(\beta_1^*)$}&\rotatebox{90}{$\bar{\U}^T_1(0.5\beta_1^*)$}&\rotatebox{90}{$\bar{\U}^T_1(2\beta_1^*)$}&  \rotatebox{90}{$\bar{\U}_1^{T}(\beta_1^*)$}&\rotatebox{90}{$\bar{\U}_2$}  &\rotatebox{90}{$\bar{\U}^{T}_{2}$}&\rotatebox{90}{\%}\\
\hline
    & $\alpha=\frac{y}{c_{\nu}}$& \multicolumn{8}{|c|}{$y=0.1$}&\multicolumn{8}{|c|}{$y=0.101$}&\multicolumn{8}{|c|}{$y=0.097$}&\raisebox{0pt}[12pt][8pt]{}\\
 \hline

\parbox[t]{4mm}{\multirow{5}{*}{\rotatebox{90}{$\beta_1^*=0.1$}}}
 &$50$  & 1.051&1.029    &1.009 &1.051&1.008&1.025&1.019  &1.029&
                1.098&1.047 &1.018&1.101&1.009&1.041&1.038 &1.048&
                1.124&1.054&1.024&1.134 &1.012&1.057&1.050&1.056&4.67 \\
               &$100$      &  1.096 &1.043&1.014&1.096&1.013& 1.043& 1.035& 1.056&
               1.189&1.076 &1.028 &1.195&1.013&1.072& 1.073  & 1.097&
               1.246&1.095&1.039&1.266&1.019 &1.105&1.098&1.112& 3.70\\
              &$250$     &1.182&   1.046& 1.019 &1.183&1.018 & 1.069&1.064 &1.118&
                                1.379& 1.097&1.038 &1.393&1.016&  1.121&1.138& 1.204&
                                1.521&1.156&1.059&1.567&1.025& 1.197& 1.197&1.241&3.12 \\
              &$500$       &1.267&1.030& 1.015 &1.268&1.016 &1.077&1.081 &1.211 &
                                 1.608&1.065&1.030 &1.639 &1.011& 1.132&1.187 & 1.372&
                                 1.928&1.116&1.051& 2.039&1.020&1.253&1.300& 1.481&2.92 \\
                &$1000$ &1.325&1.015  &1.009 &1.330&1.010 &1.061&1.073  & 1.332 &
                                 1.858&1.032&1.016 &1.923 &1.005&1.088& 1.168&1.586  &
                                 2.549&1.054& 1.025 &2.820&1.008&1.209 &1.300&1.877&2.82 \\
\hline
    & $\alpha=\frac{y}{c_{\nu}}$& \multicolumn{8}{|c|}{$y=0.2$}&\multicolumn{8}{|c|}{$y=0.202$}&\multicolumn{8}{|c|}{$y=0.193$}&\raisebox{0pt}[12pt][8pt]{}\\
    \hline

        \parbox[t]{4mm}{\multirow{5}{*}{\rotatebox{90}{$\beta_1^*=0.2$}}}
        &$50$  &1.151& 1.044   &1.016 &1.157&1.016&1.058&1.053  &1.134&
                      1.316&1.090& 1.032&1.326&1.014&1.094&1.119 &1.217&
                       1.448 &1.139&1.048&1.505&1.021 &1.149&1.177&1.288&11.95 \\
         &$100$    &1.209&1.032  &1.013  &1.221&1.015&1.066&1.068  & 1.235&
                           1.471&1.068 &1.027 &1.491&1.011&1.103&1.162  &1.377&
                           1.730& 1.120&1.045&1.848&1.018&1.186&1.266 &1.534&11.04 \\
          &$250$ &1.227&1.012&1.007 &1.247&1.008  & 1.046&1.060&1.397&
                         1.578& 1.026& 1.013 &1.614 &1.004&1.061&1.146&  1.590&
                         2.085&1.049&1.022 &2.363&1.007& 1.137 & 1.271&1.954& 10.48\\
          &$500$  &1.152&1.005& 1.003&1.178&1.003& 1.021&  1.040 &1.513 &
                          1.415&1.009&1.005 &1.447&1.002 &1.022&1.092&1.625&
                          1.945&1.017& 1.009 &2.321 &1.004 &1.051 &1.174&2.069&10.27\\
          &$1000$ &1.061&1.002&1.001 &1.083 &1.001 &1.007& 1.024 &1.586 &
                           1.145 &1.003& 1.002&1.152&1.002&1.014&1.052&1.497&
                            1.358&1.005&1.003  &1.554 &1.003 & 1.029&1.093&1.728&10.18 \\

\hline
    & $\alpha=\frac{y}{c_{\nu}}$& \multicolumn{8}{|c|}{$y=0.4$}&\multicolumn{8}{|c|}{$y=0.404$}&\multicolumn{8}{|c|}{$y=0.385$}&\raisebox{0pt}[12pt][8pt]{}\\
    \hline
    \parbox[t]{4mm}{\multirow{5}{*}{\rotatebox{90}{$\beta_1^*=0.4$}}}
        &$50$     & 1.208&1.016 &1.008 & 1.226&1.010&1.050&1.060  &1.372&
                           1.507 &1.035&1.016&1.530&1.005&1.072&1.148&1.519&
                           1.900&1.066&1.027&2.088&1.010&1.152& 1.271&1.823&34.97 \\
         &$100$     &1.151& 1.006&1.004  &1.174&1.004&1.026 & 1.041&1.499 &
                             1.399&1.013&1.007 &1.421&1.002&1.030&1.100  &1.583 &
                              1.846 &1.024&1.011&2.106 &1.004&1.071&1.196&2.006&34.18 \\
         &$250$ &1.050&1.001&1.001 & 1.066&1.001 &1.006 &1.020 & 1.598 &
                        1.128 &1.003& 1.002 &1.141&1.001 &1.011& 1.048&1.454 &
                         1.328&1.006 & 1.003& 1.491&1.003&1.028&1.091 &1.691& 33.90 \\
        &$500$ &1.013&1.000& 1.000&1.019  &1.000&  1.002& 1.011&1.633 &
                        1.024&1.001&1.001 &1.044&1.001&1.009&1.025& 1.314 &
                        1.053&1.002&1.001&1.121&1.001 &1.019&1.047 &1.373& 33.66 \\
        &$1000$ &1.003&1.000&1.000 &1.005&1.000& 1.000&  1.006& 1.649 &
                        1.003&1.000& 1.000 &1.034 &1.000 &1.006&1.014&1.208  &
                         1.005&1.000&1.000 &1.077&1.000& 1.009&1.024&1.184& 33.59 \\

\hline
\end{tabular}
\end{sidewaystable}
\end{document}